\documentclass[11pt,times,letter]{article}
\usepackage{amsmath, amsfonts,amsthm,amssymb,bm}
\usepackage{dsfont}

\usepackage{extarrows}
\usepackage[english]{babel}

\usepackage[normalem]{ulem}

\usepackage{latexsym,amssymb,amsfonts,graphicx}
\usepackage{amsmath,dsfont}
\usepackage{verbatim}
\usepackage{mathrsfs}
\usepackage{color}
\usepackage{epsfig}

\usepackage{url}
\usepackage{bm}
\usepackage{natbib}

\usepackage[colorlinks,linkcolor=blue,citecolor=blue,anchorcolor=blue]{hyperref}

\newcommand{\idc}{\mathbf{1}}
\newcommand{\Px}{ \mathbb{P} }
\newcommand{\Qx}{ \mathbb{Q} }
\newcommand{\N}{ \mathbb{N} }
\newcommand{\Ex}{ \mathbb{E} }

\def\esssup_#1{\underset{#1}{\Xi}}
\def\essinf_#1{\underset{#1}{\mathrm{ess\,inf\, }}}
\def\argmax_#1{\underset{#1}{\mathrm{arg\,max\, }}}
\def\argmin_#1{\underset{#1}{\mathrm{arg\,min\, }}}

\newcommand{\Hx}{\mathbb{H}}

\newcommand{\Fx}{\mathbb{F} }
\newcommand{\cO}{\mathcal{O}}

\newcommand{\F}{\mathcal{F}}
\newcommand{\Lip}{{\rm Lip}}

\newcommand{\R}{\mathds{R}}

\newcommand{\Asa}{{\bf(A$_{s1}$)}}
\newcommand{\Asb}{{\bf(A$_{s2}$)}}

\newcommand{\Ath}{{\bf(A$_{\Theta}$)}}

\numberwithin{equation}{section}
\newtheorem{theorem}{Theorem}[section]
\newtheorem{definition}{Definition}[section]

\newtheorem{proposition}[theorem]{Proposition}
\newtheorem{remark}[theorem]{Remark}
\newtheorem{lemma}[theorem]{Lemma}
\newtheorem{corollary}[theorem]{Corollary}

\allowdisplaybreaks[4]

\definecolor{Red}{rgb}{1.00, 0.00, 0.00}
\definecolor{DRed}{rgb}{0.5, 0.00, 0.00}

\definecolor{Blue}{rgb}{0.00, 0.00, 1.00}

\definecolor{Green}{rgb}{0.0, 0.4, 0.0}

\begin{document}

\title{Centralized systemic risk control in the interbank system: Weak formulation and Gamma-convergence}
\author{Lijun Bo \thanks{Email: lijunbo@xidian.edu.cn, School of Mathematics and Statistics, Xidian University, Xi'an 710126, P.R. China.}
\and
Tongqing Li \thanks{Email: ltqing@mail.ustc.edu.cn, School of Mathematics and Statistics, Xidian University, Xi'an 710126, P.R. China.}
\and
Xiang Yu \thanks{Email: xiang.yu@polyu.edu.hk, Department of Applied Mathematics, The Hong Kong Polytechnic University, Hung Hom, Kowloon, Hong Kong.}
}

\date{\vspace{-5ex}}

\maketitle

\begin{abstract}
This paper studies a systemic risk control problem by the central bank, which dynamically plans monetary supply to stabilize the interbank system with borrowing and lending activities. Facing both heterogeneity among banks and the common noise, the central bank aims to find an optimal strategy to minimize the average distance between log-monetary reserves of all banks and the benchmark of some target steady levels. A weak formulation is adopted, and an optimal randomized control can be obtained in the system
with finite banks by applying Ekeland's variational principle. As the number of banks grows large, we prove the convergence of optimal strategies using the Gamma-convergence argument, which yields an optimal weak control in the mean field model. It is shown that this mean field optimal control is associated to the solution of a stochastic Fokker-Planck-Kolmogorov (FPK) equation, for which the uniqueness of the solution is established under some mild conditions.

  \vspace{0.4 cm}

  \noindent{\textbf{AMS subject classifications}:\ \ 93E20, 60H30, 60F05}

  \vspace{0.2 cm}

  \noindent{\textbf{Keywords}:}\ \ Interbank system; weak formulation; mean field control; stochastic FPK equation; Gamma-convergence
\end{abstract}
\vspace{0.5cm}

\section{Introduction}

Systemic risk has been an important topic in quantitative finance and financial economics, which measures the chance that a large-scale of the interconnected system fails entirely or sequentially. It has attracted more attention after the outbreak of the global financial crisis of 2008, and abundant research works can be found in various market models by proposing different types of systemic risk measures and addressing their assessment methods, stability analysis and regulations. A comprehensive introduction to recent developments in systemic risk can be found in \cite{ForqueLangsam13}. Due to the complex risk exposures of the interacting bank network, the study of systemic risk in the interbank system has become one main focus in academic research. To describe the interacting bank system, \cite{FouqueSun13} propose the log-monetary reverse model for banks, in which the interactions among banks occur in a way that each reserve process mean-reverts to the average of the entire system with the same borrowing and lending rate.  Later, \cite{Granier-P-Y13} extend this model by allowing each bank to have two stable states, namely one normal state and one failure state. It is shown that the stronger cooperation increases the stability of the individual bank but also the overall systemic risk. \cite{BoCapponi15} study the similar systemic risk when monetary reserves are governed by a jump-diffusion system, in which the interactions among banks come from both the mean field interaction and the sensitivity of each entity via the banking sector indicator. Recently, \cite{Capponi-S-Y20} further propose a refined model when banks are organized into clusters to dynamically adjust their monetary reserves to meet some prescribed reserve requirements. In particular, to understand the systemic risk in a large interbank network, \cite{Capponi-S-Y20} characterize the asymptotic behaviour of the empirical measures of monetary reserve processes when the number of banks tends to infinity.

On the other hand, in addition to purely investigating various systemic risk measures for the interbank system, it is also of great importance to maintain the dynamic system in some steady states by managing the systemic risk actively. In practice, some financial activities such as adjustments in borrowing and lending may help to reduce the ripple effect of massive failures. Along this direction, some existing works have formulated and studied the dynamic borrowing and lending control by each component bank to minimize the systemic risk. For example, \cite{Carmona-F-S15} derive the explicit closed-loop and open-loop Nash equilibria for the mean field game (MFG) in the context of a large interbank system, based on the methodology developed in the seminal works of \cite{LasryLions07} and \cite{Huang-M-C06}. In a similar fashion, \cite{Sun18} obtains a closed-form Markov Nash equilibrium in a mean field coupled Feller's monetary reserves model with a unique idiosyncratic noise. Recently, \cite{Sun22} studies a system of heterogeneous lending and borrowing based on the relative average of log-capitalization given by the linear combination of the average within groups and ensemble average. The existence of closed-loop and open-loop Nash equilibria for the two-group case and the existence of an approximate Nash equilibrium of the MFGs in $d$ heterogeneous groups are respectively established.

As opposed to previous studies, we are interested in providing a new perspective on systemic risk control by focusing on the role of the central bank. We formulate a tractable centralized monetary supply problem by the central bank, which aims to stabilize the operation of the entire system by planning the monetary supply such as the rate of money creation or the supply of liquidity. To attain the goal of stabilization, it is assumed that the central bank initially sets up some ideal steady reverse levels for all component banks as some tracking benchmarks to exercise its monetary policy. The interactions between banks are captured by fixed borrowing and lending activities with a common noise disturbing the system. The specific systemic risk is then measured by the average distance between the log-monetary reserves and the target steady reserve levels of component banks. The mathematical task of the central bank is to adjust the monetary supply, subjecting to some quadratic control costs, such that the described systemic risk is minimized.

Moreover, motivated by the large population of commercial banks in the interbank network, we are particularly interested in the mean field model when there are infinitely many banks. It is well known that the common noise and the centralized control feature complicate the analysis of the associated McKean-Vlasov control problem; see some related work in \cite{Carmona-D-L13}, \cite{Pham16} and \cite{MottePham2021}. Furthermore, it is a challenging task to prove the strong convergence of optimal controls from the system with $N$ banks to its mean field counterpart. As a remedy, we choose to work in the weak formulation; see \cite{ElKaroui-N-J87}, \cite{Bahlali-D-M07}, \cite{AhmedCharalambous17}, \cite{Lacker16} and \cite{Barbu-R-Z18}. Our mathematical contribution is twofold. Firstly, we attempt a control problem of money supply by the central bank and establish the equivalence between the strong formulation (see problem \eqref{eq:control-problem}) and the weak formulation (see problem \eqref{eq:objectfuncRex}) with $N$ banks. In particular, we deduce the existence of an optimal weak control using the Ekeland's variational principle; see Proposition \ref{prop:ex-op-relaxed} and Propostion \ref{prop:optimal-existence}. Secondly, to examine the mean field model with infinite banks, we apply the Gamma-convergence arguments and conclude the convergence of optimizers in the Wasserstein metric (see \cite{Villani03}) from the model with $N$ heterogeneous banks when $N$ grows to infinity, which is new to the literature. The limiting weak control is shown to be the optimal solution in the mean field model that is related to the solution of a stochastic FPK equation; see Theorem \ref{thm:MainThm}. We also obtain a technical auxiliary result on the uniqueness of the solution to the stochastic FPK equation in Proposition \ref{prop:conv-3}, which guarantees the uniqueness of the limiting value functional in the mean field model.

The rest of the paper is organized as follows. In Section \ref{sec:themodel}, we first introduce the conventional strong formulation and its equivalent weak formulation of the inter-bank controlled system with finite number of banks. In Section \ref{sec:lim-ch}, we investigate the weak formulation of the mean field control problem of the sizable interbank system with infinitely many banks. In Section~\ref{sec:gamma-conv}, we show the convergence of the optimal value functionals and the associated minimizers by means of the Gamma-convergence arguments. It is proved that the limiting point is a weak optimal control in the mean field model. The proofs of some auxiliary results are collected in Appendix~\ref{app:proof1}.

\section{Model Setup and Problem Formulation}\label{sec:themodel}

\subsection{The financial model with $N$ banks}

Let $(\Omega,\F,\Px)$ be a complete probability space with the filtration $\Fx=(\F_t)_{t\in[0,T]}$ satisfying the usual conditions, where $T>0$ refers to the finite time horizon. We consider an interbank system consisting of $N$ heterogeneous banks with lending and borrowing interactions. To be precise, let us first introduce the following notations:
\begin{itemize}
    \item $a_i>0$ denotes the heterogenous constant rate of borrowing and lending of the bank $i$;
    \item $\sigma_i>0$ denotes the volatility of the idiosyncratic noise of the bank $i$;
    \item $u_i>0$ denotes the heterogenous constant intensity of monetary supply to the bank $i$;
    \item $\sigma_0\geq0$ denotes the volatility of the common noise of the interbank system;
    \item $(\theta_t)_{t\in[0,T]}$ is an $\Fx$-adapted process that captures the normalized rate of the log-monetary supply to the interbank system controlled by the central bank.
\end{itemize}
For $i=1,\ldots, N$, the log-monetary reserve of the bank $i$ in the system is assumed to satisfy the SDE that
\begin{align}\label{eq:systemmodel}
\begin{cases}
\displaystyle dX^{\theta,i}_t =  \frac{a_i}{N}\sum_{j=1}^N(X^{\theta,j}_t-X^{\theta,i}_t)dt + u_i\theta_tdt + \sigma_idW^i_t+\sigma_0dW^0_t,\quad t\in(0,T],\\[1.5em]
\displaystyle X^{\theta,i}_0=X^i_0,
\end{cases}
\end{align}
where $W^j=(W_t^j)_{t\in[0,T]}$, $j=0,1,\ldots,N$, are $N+1$ independent $\Fx$-Brownian motions and $W^0$ stands for the common noise affecting the interbank system. Similar to \cite{Carmona-F-S15} and \cite{Sun18}, the default risk is not considered in the present paper. %That is, even after reaching the default threshold of the monetary reserve, the bank still stays in the system and participates in inter-bank activities.

As documented in the existing literature (see, e.g. \cite{Blinder08}), various tools have been used by central banks to control the money supply including setting bank reserve requirements, influencing interest rates, printing money, open market operations, and etc. Among them, setting minimum reserve requirements for domestic banks has been an important instrument in regulating the money circulation and stabilizing the economy, see \cite{Gray11} and \cite{Chang2019}. On the other hand, as discussed in \cite{J12}, some unconventional tools have been employed recently in monetary policies. For example, some target steady reserve levels resulting from the target inflation rates or target interest rates have been considered in recent studies such that falling below target reserve levels triggers the injection or lending of money by the central bank, see \cite{FatoneMariani19} and \cite{Capponi-S-Y20}. In the present paper, we also consider the target steady reserve level of the bank $i$ (higher than the conventional minimum reserve level), denoted by the $\mathcal{F}_0$-measurable and square integrable random variable $Y^i$. However, for tractability, we set each target level $Y^i$ as a benchmark into the measure of systemic risk instead of incorporating it directly into the reserve dynamics as in \cite{FatoneMariani19} and \cite{Capponi-S-Y20}. As the target steady reverse levels are defined in the objective functional, the optimal money supply is exercised by the central bank such that all bank reserves evolve closely around these ideal and steady reserve levels, resulting in a desired stable economy.

Let $\mathbf{X}^{\theta,N}_t:=(X^{\theta,1}_t,\ldots,X^{\theta,N}_t)^\top$, $t\in[0,T]$, and $\mathbf{Y}^N:=(Y^1,\ldots,Y^N)^\top$. The terminal cost function $L_N:\R^N\times\R^N\to\R$ and the running cost function $R_N:\R^N\times\R^N\times\R\to\R$ are defined respectively by
\begin{equation}\label{eq:LNRN}
  L_N(\mathbf{x},\mathbf{y}):=\frac{\alpha}{N}\sum_{i=1}^NL(x_i,y_i),\quad R_N(\mathbf{x},\mathbf{y};\theta):=\frac{\beta}{N}\sum_{i=1}^NL(x_i,y_i)+\lambda \theta^2,
\end{equation}
where $\alpha,\beta,\lambda>0$ are weighting parameters of different costs, ${\bf x}:=(x_1,\ldots,x_N)^\top\in\R^N$, ${\bf y}:=(y_1,\ldots,y_N)^\top\in\R^N$, and the distance function $L(x_i,y_i):=|x_i-y_i|^2$. In particular, the term $\lambda\theta^2$ in $R_N(\mathbf{x},\mathbf{y};\theta)$ stands for the quadratic cost of the control effort. That is, the central bank aims to dynamically steer the controlled log-monetary reserves of all banks as close as possible to the benchmark steady levels when the control effort is costly. Under each control $\theta=(\theta_t)_{t\in[0,T]}$, the objective functional of the central bank is then defined by
\begin{align}\label{eq:objectfunc}
J_N(\theta)=\Ex\left[L_N(\mathbf{X}^{\theta,N}_T,\mathbf{Y}^N)+\int_0^T R_N(\mathbf{X}^{\theta,N}_t,\mathbf{Y}^N;\theta_t)dt\right].
\end{align}

%To depict a large population of banks arising from the real life applications, we are particularly interested in the sizable interbank system when the number $N$ grows large, which naturally leads to a mean field type of control problem by the central bank.

For the future formulation in the limiting model, let us define the type vector of the interbank system \eqref{eq:systemmodel} by
\begin{align}\label{eq:parameters}
\xi^i:=(a_i,u_i,\sigma_i)^\top\in\mathcal{O}:=\R_+^3.
\end{align}
Let $\Theta\subset\R$ be the given policy space. We denote $\mathbb{H}^2$ the space of all $\mathbb{F}$-adapted and real-valued processes $\theta=(\theta_t)_{t\in[0,T]}$ satisfying $\Ex[\int_0^T|\theta_t|^2dt]<\infty$. The set of admissible controls, denoted by $\mathbb{U}^{\Px,\Fx}$, contains all processes $\theta\in \mathbb{H}^2$ such that $\theta_t\in\Theta$, $\Px(d\omega)\otimes dt\text{-a.s.}$. We also endow the set $\mathbb{U}^{\Px,\Fx}$ with the norm $\|\theta\|_{\Hx^2}:=\sqrt{\Ex[ \int_0^T|\theta_t|^2dt]}$.
%({\color{red}Moved the definition here. Do we need this integrability to define the admissible control at the beginning? In particular, we may need this to define the BSDE problem.})
%\ \\
%{\color{blue}To be compatible with the following setting that $\theta$ is a $\mathcal{L}_T^2$-valued random variable, we'd better move the above definition here, although the square integrability is not necessary as $\Theta$ is compact. Thus, the following definition can be deleted.}
%The set of admissible controls is then defined by
%\begin{equation}\label{eq:Uad}
%  \mathbb{U}^{\Px,\Fx}:=\{\theta=(\theta_t)_{t\in[0,T]}:~\theta\ \text{is}~\Fx\text{-adapted~and}~ \theta_t\in\Theta,~\Px(d\omega)\otimes dt\text{-a.s.}\}.
%\end{equation}
Given the admissible control set $\mathbb{U}^{\Px,\Fx}$, the centralized optimal stabilization problem by the central bank is given by
\begin{align}\label{eq:control-problem}
\inf_{\theta\in\mathbb{U}^{\Px,\Fx}}J_N(\theta).
\end{align}
We refer \eqref{eq:control-problem} as the strong formulation of the optimal %mean-field
control problem. %, and its well-posedness will be addressed in the next subsection.
The following assumptions are imposed throughout the paper:
\begin{itemize}
\item[{\Asa}] There exists a global finite constant $K$ such that $|\xi^i| \leq K$ for all $i\geq1$, and $\sup_{i\in\N}\Ex[|X_0^i|^{2+\varrho}]\leq K$ for some $\varrho>0$.
%There exists a global constant $K$ such that $|\xi_i| \leq K$ for all $i\geq1$.
\item[{\Ath}] The policy space $\Theta\subset\R$ is a (nonempty) compact and convex set.
\end{itemize}

\begin{remark}
The motivation of the problem \eqref{eq:control-problem} is to provide a new perspective on the interbank systemic risk, featuring the authority role of the central bank in stabilizing the operation with some benchmark reserve targets. However, it is well understood that the monetary policy by the central bank is a challenging financial and economic issue involving many aspects of the domestic market and sophisticated decisions by commercial banks. The simple model and objective functional in \eqref{eq:control-problem} only serve as a first step to understand the decision making by the central bank and develop some tools to investigate more realistic models and other systemic risk measures in future studies. For example, by incorporating active borrowing and lending decisions by component banks, we may formulate a Stackelberg mean field game problem and generalize the methodology in the present paper when the central bank can intervene the interactions of component banks to keep the domestic economy stable and growing.
\end{remark}

\subsection{Optimal control under strong formulation}\label{subsec:op-con}

This section establishes the well-posedness of the strong optimal %mean-field
control problem \eqref{eq:control-problem} with finite banks. Let us first rewrite the system \eqref{eq:systemmodel} in a compact form that
\begin{equation}\label{eq:finite-x}
 d\mathbf{X}^{\theta}_t=\mathbf{b}(\mathbf{X}^{\theta}_t,\theta_t)dt + \Sigma d\mathbf{W}_t^0,\ \ t\in[0,T],
\end{equation}
where $\mathbf{X}_t^{\theta} =(X^{\theta,1}_t,\ldots,X^{\theta,N}_t)^\top$, $\mathbf{W}_t^0 := (W_t^0,W_t^1,\ldots,W_t^N)^\top$, and the controlled drift term is given by $ \mathbf{b}(\mathbf{X}_t^{\theta},\theta_t):=A\mathbf{X}_t^{\theta}+\theta_t\mathbf{u}$, where
{\small
\begin{equation}\label{eq:coeff-b}\begin{aligned}
A:=& \frac{1}{N}\left[\begin{array}{cccc}
   (1-N)a_1 &    a_1   & \cdots &    a_1\\
      a_2   & (1-N)a_2 & \cdots &    a_2\\
    \vdots  &  \vdots  & \ddots &  \vdots\\
      a_N   &    a_N   & \cdots & (1-N)a_N
 \end{array}\right],
\end{aligned}
\end{equation}
}
and ${\bf u}:=(u_1,\ldots,u_N)^{\top}$. The volatility matrix is given by
\begin{align}\label{eq:volSigma}
  \Sigma:= \left[
  \begin{array}{ccccc}
    \sigma_0 & \sigma_1 &    0     & \cdots & 0  \\
    \sigma_0 &    0     & \sigma_2 & \cdots & 0\\
    \vdots   & \vdots   & \vdots   & \ddots &\vdots \\
    \sigma_0 &    0     &    0     & \cdots & \sigma_N
  \end{array}
\right]_{N\times(N+1)}.
\end{align}
For ease of presentation, we omit the dependence of $A,{\bf b},{\bf u},\Sigma,{\bf X}^{\theta},{\bf W}^0$ on the number $N$.

We first present some technical results on moment estimations of the state process and its proof is given in Appendix \ref{app:proof1}.
\begin{lemma}\label{lem:estimates}
Let assumptions {\Asa} and {\Ath} hold. For any $\delta>0$, there exists a constant $D_1>0$ depending on $K$, $\Theta$ and $p$ only such that
\begin{equation}\label{eq:x-bound}
  \sup_{\theta\in\mathbb{U}^{\Px,\Fx}}\sup_{i\in\N}\Ex\left[\sup_{t\in[0,T]}\left|X_t^{\theta,i}\right|^{2+\varrho}\right]\leq D_1,\quad \lim_{\delta\to0}\sup_{\theta\in\mathbb{U}^{\Px,\Fx}}\sup_{i\in\N}\Ex\left[ \sup_{|t-s|\leq\delta} \left|X^{\theta,i}_t-X^{\theta,i}_s\right|^{2} \right]=0,
\end{equation}
where $\varrho>0$ is given in the assumption {\Asa}.
\end{lemma}
As a preparation, let us introduce a parameterized HJB equation of $V(t,\mathbf{x};\mathbf{y})$ for $(t,\mathbf{x})\in[0,T]\times\R^N$ and the parameter $\mathbf{y}\in\R^N$ that
\begin{equation}\label{eq:HJB}
\left\{\begin{aligned}
  &-\frac{\partial V}{\partial t}(t,\mathbf{x};\mathbf{y}) -\mathcal{H}(t,\mathbf{x},\nabla_{\mathbf{x}}V(t,\mathbf{x};\mathbf{y}), \nabla_{\mathbf{x}}^2V(t,\mathbf{x};\mathbf{y});\mathbf{y})=0,\\[0.4em]
  &~V(T,\mathbf{x};\mathbf{y})=L_N(\mathbf{x},\mathbf{y}),
\end{aligned}\right.
\end{equation}
where the parameterized Hamiltonian is defined by, for $(t,\mathbf{x})\in[0,T]\times\R^{N}$, $\mathbf{p}\in\R^N$, $M\in\R^{N\times N}$ and the parameter  $\mathbf{y}\in\R^N$,
\[\mathcal{H}(t,\mathbf{x},\mathbf{p},M;\mathbf{y}) :=\inf_{\theta\in\Theta}\left\{\mathbf{b}(\mathbf{x},\theta)^{\top}\mathbf{p} +\frac{1}{2}\mathrm{tr}(\Sigma\Sigma^\top M)+R_N(\mathbf{x},\mathbf{y};\theta)\right\}.\]
The next auxiliary result gives the well-posedness of the HJB equation, and its proof is deferred to Appendix \ref{app:proof1}.

\begin{lemma}\label{lem:wellpo-HJB}
Let assumptions {\Asa} and {\Ath} hold. For any $\mathbf{y}\in\R^N$, the parameterized HJB equation \eqref{eq:HJB} has a unique classical solution $V(\cdot,\cdot;\mathbf{y})\in C^{1,2}([0,T)\times\R^N)\cap C([0,T]\times\R^N)$. Moreover, it also satisfies the quadratic growth condition, i.e., there exists $D_2>0$ such that $|V(t,\mathbf{x};\mathbf{y})|\leq D_2(1+|\mathbf{x}|^2)$.
\end{lemma}

By applying Remark IV.3.2 (Chapter IV) of \cite{FlemingSoner06} with Lemma \ref{lem:wellpo-HJB} and the assumption {\Ath}, there exists a continuous function $f^*:[0,T]\times\R\to\Theta$ such that
\begin{align}\label{eq:optimalMar}
f^*(t,\mathbf{x};\mathbf{y})&:=\mathop{\arg\min}_{\theta\in\Theta}\left\{\mathbf{b}(\mathbf{x},\theta)^{\top}\nabla_{\mathbf{x}}V(t,\mathbf{x};\mathbf{y}) +\frac{1}{2}\mathrm{tr}(\Sigma\Sigma^\top \nabla_{\mathbf{x}}^2V(t,\mathbf{x};\mathbf{y}))+R_N(\mathbf{x},\mathbf{y};\theta)\right\}\nonumber\\
&=\Pi_{\Theta}\left(-\frac{1}{2\lambda}\sum_{j=1}^{N}u_j\frac{\partial V(t,\mathbf{x};\mathbf{y})}{\partial x_j}\right),\quad  (t,\mathbf{x})\in[0,T]\times\R^N,
\end{align}
where $V(t,\mathbf{x};\mathbf{y})$ is the unique classical solution of the parameterized HJB equation \eqref{eq:HJB}, and $\Pi_\Theta$ denotes the usual (single-valued) Euclidean projection, i.e., for any $x\in\R$, $\Pi_{\Theta}(x)$ is the unique point such that $|x-\Pi_{\Theta}(x)|=d_{\Theta}(x):=\inf_{y\in\Theta}|x-y|$.

The next result readily follows.
\begin{lemma}\label{lem:ex-op-strict}
Let $f^*(t,\mathbf{x};\mathbf{y})$ be given by \eqref{eq:optimalMar}. Then, the following SDE:
\begin{equation}\label{eq:Op-SDE}
\mathbf{X}^{*}_t=\mathbf{X}_0+\int_0^t\mathbf{b}(\mathbf{X}^{*}_s,f^*(s,\mathbf{X}^{*}_s;\mathbf{Y}))ds + \Sigma \mathbf{W}^0_t,~~t\in[0,T]
\end{equation}
has a unique strong solution. Moreover, let $\theta^*_t=f^*(t,\mathbf{X}^{*}_t;\mathbf{Y})$ for $t\in[0,T]$, then $\theta^*=(\theta^*_t)_{t\in[0,T]}\in\mathbb{ U}^{\Px,\Fx}$ is the optimal control of \eqref{eq:control-problem}, i.e., for any $\theta\in\mathbb{U}^{\Px,\Fx}$, it holds that
\[\Ex[V(0,\mathbf{X}_0;\mathbf{Y})]=J_N(\theta^*)\leq J_N(\theta).\]
\end{lemma}

Lemma \ref{lem:ex-op-strict} establishes the existence of the optimal control of \eqref{eq:control-problem} for each $N\geq 1$ via the classical solution of the parameterized HJB equation \eqref{eq:HJB}. Our next goals are to study $(i)$ the convergence of optimal controls as $N\to\infty$; $(ii)$ the limit is the minimizer of mean field control problem. A natural observation is to apply \eqref{eq:optimalMar} to establish thus limit as $N\to\infty$. However, we note that, the feedback control function heavily depends on the dimension $N$ coming from $\mathbf{x},\mathbf{y}$ and $V$. Therefore, it is usually a challenging task to build the rigorous connection between the problem with $N$ banks and the mean field control problem by proving the strong convergence of optimal control processes from the system with $N$ banks to the mean field counterpart. As a detour, instead of studying the control problem \eqref{eq:control-problem} directly, we focus on an equivalent weak formulation by looking for the optimal solution in an appropriate space that attains the same value function. Some notable advantages of the weak formulation reside in that: $(i)$ the convenient compactness property and weak convergence arguments that can significantly simplify the proof of the existence of an optimal solution; $(ii)$ the appropriate theoretical setting to align with the Gamma-convergence arguments (see, e.g. \cite{Braides14}), in which we are allowed to rigorously show the convergence of the optimizers under the Wasserstein metric as $N$ tends to infinity, and the limit point corresponds to an optimal solution in the model with infinitely many banks.

\subsection{Optimal control under weak formulation}

We start from the canonical space representation. Let  $\Xi:=\R^{2}$ and $\mathcal{C}_T:=C([0,T];\R)$. The canonical space is defined by an infinite product space that
\begin{align}\label{eq:OmegaNFN}
\Omega_\infty:= \Xi^{\mathbb{N}}\times\mathcal{C}_T\times \mathcal{C}_T^{\mathbb{N}}\times\mathcal{L}_T^2,\quad \F_\infty:=\mathcal{B}(\Omega_\infty),
\end{align}
where $\mathcal{L}_T^2:=L^2([0,T];\R)$, and $\mathcal{B}(\Omega_\infty)$ is the Borel sigma-algebra generated by open sets in $\Omega_\infty$. The product space $\Omega_{\infty}$ is endowed with the metric that
\begin{align}\label{eq:dinfty}
d((\gamma,\varsigma,w,\kappa),(\hat{\gamma},\hat{\varsigma},\hat{w},\hat{\kappa}))
:=d_{1}(\gamma,\hat{\gamma})+d_2(\varsigma,\hat{\varsigma})+d_3(w,\hat{w})+d_4(\kappa,\hat{\kappa}),
\end{align}
for any $(\gamma,w,\varsigma,\rho)$, $(\hat{\gamma},\hat{w},\hat{\varsigma},\hat{\rho})\in \Omega_{\infty}$. Here, the metrics $d_i$, $i=1,\ldots, 4$, are defined respectively by
\begin{align}\label{eq:metricd1234}
d_1(\gamma,\hat{\gamma})&:=\sum_{i=1}^{\infty}2^{-i}\frac{|\gamma_i-\hat{\gamma}_i|}{1+|\gamma_i-\hat{\gamma}_i|},\quad \gamma=(\gamma_i)_{i=1}^{\infty},\ \hat{\gamma}=(\hat{\gamma}_i)_{i=1}^{\infty}\in \Xi^{\mathbb{N}};\nonumber\\
d_2(\varsigma,\hat{\varsigma})&:=\|\varsigma-\hat{\varsigma}\|_T=\sup_{t\in[0,T]}\left|\varsigma_t-\hat{\varsigma}_t\right|,\quad \varsigma,\hat{\varsigma}\in\mathcal{C}_T;\\
d_3(w,\hat{w})&:=\sum_{i=1}^{\infty}2^{-i}\frac{\|w_i-\hat{w}_i\|_{T}}{1+\|w_i-\hat{w}_i\|_{T}},\quad w=(w_i)_{i=1}^{\infty},\ \hat{w}=(\hat{w}_i)_{i=1}^{\infty}\in \mathcal{C}_T^{\N};\nonumber\\
d_4(\kappa,\hat{\kappa})&:=\|\kappa-\hat{\kappa}\|_{\mathcal{L}_T^2}=\left(\int_0^T\left|\kappa_t-\hat{\kappa}_t\right|^2dt\right)^{\frac{1}{2}},\quad \kappa,\hat{\kappa}\in\mathcal{L}^2_T.\nonumber
\end{align}
We also define ${\bm\zeta}:=(\zeta^1,\zeta^2,\ldots)^\top$ with $\zeta^i:=(X_0^i,Y^i)\in\Xi$ and ${\bf W}:=(W^1,W^2,\ldots)^\top$. As a result, for any $\theta\in\mathbb{U}^{\Px,\Fx}$,  $({\bm\zeta},(W^0,{\bf W}),\theta)$ is an $\Omega_{\infty}$-valued random variable under $(\Omega,\F,\Px)$. Moreover, we use the notation $\widehat{\cal X}:=(\widehat{\bm\zeta},(\widehat{W}^0,\widehat{\bf W}),\widehat{\theta})$ as the coordinate process on $\Omega_{\infty}$, i.e., $\widehat{\cal X}(\omega)=\omega$ for all $\omega\in\Omega_{\infty}$. Let $\widehat{\Fx}=(\F^{\widehat{\cal X}}_t)_{t\in[0,T]}$ be the natural filtration generated by the coordinate process $\widehat{\cal X}$.
%We also denote $({\bm\zeta},(W^0,{\bf W}),\theta)$, where ${\bm\zeta}=(\zeta^1,\zeta^2,\ldots)^\top$ and ${\bf W}=(W^1,W^2,\ldots)^\top$,
%as the identity mapping on $\Omega_{\infty}$ with $\zeta^i=(X_0^i,Y^i)\in\Xi_K$, $W^0,W^i\in{\cal C}_T$ and $\theta\in{\cal L}_T^2$. That is, for $\omega\in\Omega_{\infty}$, we have $({\bm\zeta}, (W^0, {\bf W}),\theta)(\omega)=\omega$ on $[0,T]$, i.e., $({\bm\zeta},(W^0,{\bf W}),\theta)$ is a coordinate process. Define $\Fx=(\F^{\zeta,W,\theta}_t)_{t\in[0,T]}$ as the natural filtration generated by $({\bm\zeta},(W^0,\mathbf{W}),\theta)$, i.e., $\F_t^{\zeta,W,\theta}:=\sigma(\bm{\zeta})\vee\sigma(W^0_s,{\bf W}_s,\theta_s;s\leq t)$ for $t\in[0,T]$.

Given the coordinate process $\widehat{\cal X}=(\widehat{\bm\zeta},(\widehat{W}^0,\widehat{\bf W}),\widehat{\theta})$, we are ready to give the definition of admissible controls in our framework.
%{\red with $\sup_{i\geq 1}\int_{\R}|x_i|^{2+\varrho}+|y_i|^2d\nu<\infty$}
\begin{definition}\label{def:relax-sol}
For $\nu\in{\cal P}(\Xi^{\mathbb{N}})$ (i.e., the space of probability measures on $\Xi^{\mathbb{N}}$), we call $\mathcal{Q}(\nu)$ the set of admissible controls as a set of probability measures $Q$ on $(\Omega_{\infty},\F_{\infty})$ satisfying
\begin{description}
\item[(i)] $Q\circ\widehat{\bm\zeta}^{-1}=\nu$;
\item[(ii)] $(\widehat{W}^0,\widehat{\mathbf{W}})$ is a sequence of independent Wiener processes on $(\Omega_{\infty},\widehat{\F}_{\infty},\widehat{\Fx},Q)$;
\item[(iii)] $\widehat{\theta}\in\mathbb{U}^{Q,\widehat{\Fx}}$.
\end{description}
\end{definition}
In view of Definition~\ref{def:relax-sol}, we have that for any $Q\in{\cal Q}(\nu)$, there exists a coordinate process $(\widehat{\bm\zeta},(\widehat{W}^0,\widehat{\bf W}),\widehat{\theta})$ under $(\Omega_{\infty},\widehat{\F}_{\infty},\widehat{\Fx},Q)$ satisfying \textbf{(i)}-\textbf{(iii)} in Definition~\ref{def:relax-sol}. For $\nu\in{\cal P}(\Xi^{\mathbb{N}})$, the targeted {\it weak control problem} associated to the strong control problem \eqref{eq:control-problem} is formulated as:
\begin{align}\label{eq:objectfuncRex}
V_N^R(\nu)=\inf_{Q\in{\cal Q}(\nu)}J_N^R(Q):=\inf_{Q\in{\cal Q}(\nu)}\Ex^Q\left[L_N(\widehat{\mathbf{X}}^{\theta, N}_T,\widehat{\mathbf{Y}}^N)+\int_0^T R_N(\widehat{\mathbf{X}}^{\theta, N}_t,\widehat{\mathbf{Y}}^N;\widehat{\theta}_t)dt\right],
\end{align}
where $\Ex^Q$ is the expectation operator under $Q\in{\cal Q}(\nu)$. For $i\geq1$, the state process $\widehat{X}^{\theta,i}$ follows the SDE \eqref{eq:systemmodel} driven by the coordinate process $(\widehat{\bm\zeta},(\widehat{W}^0,\widehat{\mathbf{W}}),\widehat{\theta})$ with the initial date $\widehat{\zeta}^i=(\widehat{X}_0^i,\widehat{Y}^i)\in\Xi$.  Note that the weak control strategy actually corresponds to the mixed (or randomized) strategy in the context of Markov decision process (see, e.g. \cite{Bertsekas05} and \cite{HernandezLermaLasserre96}).

We now present the next important result on the existence of the optimal weak control and the equivalence between two problem formulations.

\begin{proposition}\label{prop:ex-op-relaxed}
Let assumptions {\Asa} and {\Ath} hold. For any $\nu\in{\cal P}(\Xi^{\N})$ with $\sup_{i\in\N}\int_{\Xi^{\N}}(|x_i|^{2+\varrho}+|y_i|^2)\nu(d({\bf x},{\bf y}))<\infty$, there exists $\widehat{Q}^*\in{\cal Q}(\nu)$ attaining the value function that
\begin{align}\label{eq:equione00}
V_N^R(\nu)=\inf_{Q\in{\cal Q}(\nu)}J_N^R(Q)=J_N^R(\widehat{Q}^*).
\end{align}
Moreover, the value function in the strong control problem \eqref{eq:control-problem} coincides with the value function in the weak control problem \eqref{eq:objectfuncRex} that
\begin{align}\label{equitwo}
\inf_{Q\in{\cal Q}(\nu)}J_N^R(Q)=\Ex[V(0,\mathbf{X}_0;\mathbf{Y})]=\inf_{\theta\in\mathbb{U}^{\Px,\Fx}}J_N(\theta).
\end{align}
\end{proposition}

\begin{proof}

Recall that ${\bm\zeta}=(\zeta^{i})_{i=1}^{\infty}$ is a $\Xi^{\N}$-valued r.v. with the law $\Px\circ{\bm\zeta}^{-1}=\nu$, and $(W^0,{\bf W})$ is a sequence of independent Brownian motions under the original probability space $(\Omega,\F,\Px)$.
We then define $\widehat{Q}^*:=\Px\circ({\bm\zeta},(W^0,{\bf W}),\theta^*)^{-1}$ and check that $\widehat{Q}^*\in{\cal Q}(\nu)$. To this end, let $\widehat{\cal X}^*:=(\widehat{\bm\zeta}^*,(\widehat{W}^{*,0},{\bf \widehat{W}}^*),\widehat{\theta}^*)$ with $\widehat{\bm\zeta}^*=(\widehat{\zeta}^{*,1},\widehat{\zeta}^{*,2},\ldots)^\top$ and $\widehat{\zeta}^{*,i}:=(\widehat{X}_0^{*,i},\widehat{Y}^{*,i})$ be the coordinate process on $\Omega_{\infty}$, i.e., $\widehat{\cal X}^*(\omega)=\omega$ for all $\omega\in\Omega_{\infty}$. Moreover, let $\widehat{\Fx}^*:=(\F^{\widehat{\cal X}^*}_t)_{t\in[0,T]}$ be the natural filtration generated by this coordinate process. From the fact $\theta^*\in\mathbb{U}^{\Px,\Fx}$,  it is easy to see that
\begin{description}
  \item[(i)] $\widehat{Q}^*\circ(\widehat{{\bm\zeta}}^*)^{-1}=\Px\circ({\bm\zeta})^{-1}=\nu$.
  \item[(ii)] $(\widehat{W}^{*,0},\widehat{\bf W}^*)$ are independent Brownian motions on $(\Omega_{\infty},\widehat{\F}_{\infty}^*,\widehat{\Fx}^*,\widehat{Q}^*)$.
  \item[(iii)] $\widehat{\theta}^*\in\mathbb{U}^{\widehat{Q}^*,\widehat{\Fx}}$.
\end{description}
Furthermore, let $\widehat{\mathbf{X}}^*$ be the solution to the SDE \eqref{eq:finite-x} under $(\Omega_{\infty},\widehat{\F}_{\infty},\widehat{\Fx}^*,\widehat{Q}^*)$ w.r.t. the coordinate process $(\widehat{\bm\zeta}^*,(\widehat{W}^{*,0},{\bf \widehat{W}}^*),\widehat{\theta}^*)$. By the same arguments above, we can show that
\[\Ex^{\widehat{Q}^*}[V(T,\widehat{\mathbf{X}}^*_T;\widehat{\mathbf{Y}}^{*})]=\Ex^{\widehat{Q}^*}[V(0,\widehat{\mathbf{X}}_0^*;\widehat{\mathbf{Y}}^{*})] -\Ex^{\widehat{Q}^*}\left[\int_0^TR_N(\widehat{\mathbf{X}}^*_t,\widehat{\mathbf{Y}}^*;\widehat{\theta}_t^*)dt\right].\]
By the definition of ${\cal Q}(\nu)$, we get that $\Px\circ(\mathbf{X}_0,\mathbf{Y})^{-1}= \widehat{Q}^*\circ(\widehat{\mathbf{X}}_0^*,\widehat{\mathbf{Y}}^{*})^{-1}$, thus
\[\Ex[V(0,\mathbf{X}_0;\mathbf{Y})]=\Ex^{\widehat{Q}^*}[V(0,\widehat{\mathbf{X}}_0^*;\widehat{\mathbf{Y}}^{*})]=J_N^R(\widehat{Q}^*).\]
Similarly, for any $Q\in{\cal Q}(\nu)$, let $\widehat{\mathbf{X}}^{\theta}$ be the process satisfying the SDE \eqref{eq:finite-x} under $(\Omega_{\infty},\widehat{\F}_{\infty},\widehat{\Fx},Q)$ with respect to the coordinate process $\widehat{\cal X}=(\widehat{\bm\zeta},(\widehat{W}^0,{\bf \widehat{W}}),\widehat{\theta})$, where $\widehat{\Fx}:=(\F^{\widehat{\cal X}}_t)_{t\in[0,T]}$. It can be deduced from \eqref{eq:HJB} and the definition of ${\cal Q}(\nu)$ that
\begin{align*}
J_N^R(\widehat{Q}^*)=\Ex[V(0,\mathbf{X}_0;\mathbf{Y})]=\Ex^Q[V(0,\widehat{\mathbf{X}}_0;\widehat{\mathbf{Y}})]\leq J_N^R(Q),
\end{align*}
where the above inequality $\Ex^Q[V(0,\widehat{\mathbf{X}}_0;\widehat{\mathbf{Y}})]\leq J_N^R(Q)$ can be derived by the similar argument used in the proof of \eqref{eq:V-leq}.
%{\color{blue}(The proof of the above inequality is completely the same as the one of $\Ex[V(0,\mathbf{X}_0;\mathbf{Y})]\leq J_N(\theta)$. Concretely, define
%\[\hat{\tau}_m:=\inf\left\{t\in[0,T]:\int_0^t|\nabla_{\mathbf{x}}V(t,\widehat{\mathbf{X}}^{\theta}_s;\widehat{\mathbf{Y}})^\top\Sigma|^2ds >m\right\},\quad m>0.\]
%Then by It\^{o}'s formula, we have
%\[\begin{aligned}
%\Ex^Q[L_N(\widehat{\mathbf{X}}^{\theta}_T,\widehat{\mathbf{Y}})]\leftarrow\Ex^Q[V(\hat{\tau}_m,\widehat{\mathbf{X}}^{\theta}_{\tau_m};\widehat{\mathbf{Y}})]
%&\geq\Ex^Q[V(0,\widehat{\mathbf{X}}_0;\widehat{\mathbf{Y}})] -\Ex\left[\int_0^{\hat{\tau}_m}R_N(\widehat{\mathbf{X}}^\theta_t,\widehat{\mathbf{Y}};\hat{\theta}_t)dt\right]\\
%&\geq\Ex^Q[V(0,\widehat{\mathbf{X}}_0;\widehat{\mathbf{Y}})] -\Ex^Q\left[\int_0^TR_N(\widehat{\mathbf{X}}^\theta_t,\widehat{\mathbf{Y}};\hat{\theta}_t)dt\right],
%\end{aligned}\]
%i.e. $\Ex^Q[V(0,\widehat{\mathbf{X}}_0;\widehat{\mathbf{Y}})]\leq J_N^R(Q)$.)\\}
Thus, we have the desired equivalence that
\[J_N^R(\widehat{Q}^*)=V_N^R(\nu)=\inf_{Q\in{\cal Q}(\nu)}J_N^R(Q)=\Ex[V(0,\mathbf{X}_0;\mathbf{Y})]=\inf_{\theta\in\mathbb{U}^{\Px,\Fx}}J_N(\theta),\]
which completes the proof.
\end{proof}

\section{Weak Control Problem in the Mean Field Model}\label{sec:lim-ch}
This section establishes, respectively, the convergence results of the empirical processes and the cost functionals under weak formulation from the system with finite banks to the mean field model.
\subsection{Convergence of the empirical processes}

To characterize the limiting behavior of the objective functional $J_N^R$ (see \eqref{eq:objectfuncRex}) in the mean field sense, we first show the convergence of the sequence of empirical processes. Let $Q_N\in\mathcal{Q}(\nu)$ and $(\widehat{\bm\zeta}^N,(\widehat{W}^{N,0},\widehat{\mathbf{W}}^N),\widehat{\theta}^N)$ be the corresponding coordinate process to $Q_N$ as in Definition \ref{def:relax-sol}. Also, we consider $\widehat{\mathbf{X}}^N:=(\widehat{X}^{N,1}_t,\ldots,\widehat{X}^{N,N}_t)^\top_{t\in[0,T]}$ that solves the following system:
\begin{equation}
  \left\{\begin{aligned}
    d\widehat{X}^{N,i}_t &= \frac{a_i}{N}\sum_{j=1}^N(\widehat{X}^{N,j}_t-\widehat{X}^{N,i}_t)dt + u_i\widehat{\theta}_t^Ndt + \sigma_id\widehat{W}^{N,i}_t+\sigma_0d\widehat{W}^{N,0}_t,\\
    \widehat{X}^{N,i}_0&=\widehat{X}^i_0,\qquad i=1,\ldots,N.
  \end{aligned}\right.
\end{equation}
Here, for ease of notation, we omit the dependence on $\widehat{\theta}^N$ in $\widehat{\mathbf{X}}^{N}$. For a metric space $(\mathcal{X},d)$ and $p\geq0$, let $\mathcal{P}_p(\mathcal{X})$ be the set of probability measures with finite moments of order $p>0$, i.e., $\mathcal{P}_p(\mathcal{X}):=\left\{\mu\in\mathcal{P}(\mathcal{X}):~ \int_{\mathcal{X}}d^p(x,x_0)\mu(dx)<\infty\right\}$, for some $x_0\in\mathcal{X}$.

Correspondingly, we denote by ${\cal W}_{{\cal X},p}(\cdot,\cdot)$ the $p$-th order Wasserstein metric on ${\cal X}$. Let $E:=\cO\times\R\times\R$, and define the $\mathcal{P}_2(E)$-valued process by
\begin{equation}\label{eq:empirical-finite}
\mu^N_t:=\frac{1}{N}\sum_{i=1}^N\delta_{(\xi^i,\widehat{Y}^{N,i},\widehat{X}^{N,i}_t)},\quad t\in[0,T].
\end{equation}
We shall see that, for $N\geq1$, $\mu^N=(\mu^N_t)_{t\in[0,T]}$ is a sequence in $S:=C([0,T];\mathcal{P}_2(E))$, which is the space of continuous $\mathcal{P}_2(E)$-valued functions on $[0,T]$. We equip this space with the uniform metric $d_{S}$ defined by
\begin{equation}\label{eq:metric-S}
  d_S(\rho,\hat{\rho}):=\sup_{t\in[0,T]}\mathcal{W}_{E,2}(\rho_t,\hat{\rho}_t),\quad \rho,\hat{\rho}\in S.
\end{equation}
We also consider the joint distribution that
\begin{align}\label{eq:bbQN}
  \Qx^N:=Q_N\circ(\mu_0^N,\widehat{\theta}^N,\mu^N)^{-1}.
\end{align}
The first result of this section is to characterise the limit of $\Qx^N$ under assumptions {\Asa} and {\Ath}, which is related to the solution of a stochastic FPK equation. To begin with, for any $\phi$ in $C_b^2(\R)$, we introduce the infinitesimal generator that
\begin{align}\label{eq:generator}
  \mathcal{A}^{m,\theta}\phi(x):=g_{a,u}^{m,\theta}(x)\phi'(x)+\frac{\sigma^2+\sigma_0^2}{2}\phi''(x),\quad x\in\R,~ m\in\mathcal{P}_2(E),~\theta\in\Theta,
\end{align}
where $g_{a,u}^{m,\theta}(x):= a\left(\int_{E} zm(d\xi,dy,dz)-x\right)+u\theta$.

Moreover, let $\hat{S}:=\mathcal{P}_2(E)\times\mathcal{L}_T^2\times S$. We endow the space $\hat{S}$ with the following metric that
\begin{align}\label{eq:metric-hatS}
d_{\hat{S}}((\nu,\theta,\rho),(\hat{\nu},\hat{\theta},\hat{\rho})):={\cal W}_{E,2}(\nu,\hat{\nu})+\|\theta-\hat{\theta}\|_{{\cal L}_T^2}+d_S(\rho,\hat{\rho}),
\end{align}
for $(\nu,\theta,\rho)$ and $(\hat{\nu},\hat{\theta},\hat{\rho})\in\hat{S}$. We then have the next result.
\begin{theorem}\label{thm:convergence-empirical}
Let assumptions {\Asa} and {\Ath} hold. We assume in addition that %$Q_N\circ(\mu_0^N,\theta^N)^{-1}\Rightarrow\nu_0$
\begin{equation}\label{eq:convercondThm4-1}
Q_N\circ(\mu_0^N,\widehat{\theta}^N)^{-1}\Rightarrow\nu_0,
\end{equation}
as $N\to\infty$, for some $\nu_0\in\mathcal{P}(\mathcal{P}_2(E)\times\mathcal{L}^2_T)$. Then $\{\Qx^N\}_{N\geq 1}$ is relatively compact in $\mathcal{P}_2(\hat{S})$. Furthermore, if the law of an $\hat{S}$-valued r.v. $(\tilde{\mu}_0,\tilde{\theta},\tilde{\mu})$ defined on some probability space $(\widetilde{\Omega},\widetilde{\mathcal{F}},\widetilde{\Px})$ is the limit of a convergent subsequence of $\{\Qx^N\}_{N\geq 1}$, then $\tilde{\mu}=(\tilde{\mu}_t)_{t\in[0,T]}$ satisfies the stochastic FPK equation that
\begin{equation}\label{eq:SFPK}
\langle \tilde{\mu}_t,\phi \rangle=\langle \tilde{\mu}_0,\phi \rangle+\int_0^t \langle \tilde{\mu}_s,\mathcal{A}^{\tilde{\mu}_s,\tilde{\theta}_s}\phi \rangle ds+\sigma_0\int_0^t \langle \tilde{\mu}_s,\phi' \rangle d\widetilde{W}^0_s,\quad \forall~\phi\in C_b^2(\R),
\end{equation}
where $\widetilde{W}^0=(\widetilde{W}^0_t)_{t\in[0,T]}$ is a standard Brownian motion under $(\widetilde{\Omega},\widetilde{\mathcal{F}},\widetilde{\Px})$ and
\begin{align*}
        \langle \tilde{\mu}_t,\mathcal{A}^{\tilde{\mu}_t,\tilde{\theta}_t}\phi \rangle=\int_E \left[g_{a,u}^{\tilde{\mu}_t,\theta}(x)\phi'(x)+\frac{\sigma^2+\sigma_0^2}{2}\phi''(x)\right]\tilde{\mu}_t(d(a,u,\sigma),dy,dx).
      \end{align*}
\end{theorem}

\begin{remark}
In the special case when $\sigma_0=0$, the system only has idiosyncratic noises that is similar to the model in \cite{Bo-C-L20}. The resulting FPK equation is deterministic and easy to handle. However, with the presence of the common noise, the equation becomes a more challenging nonlinear stochastic PDE. One main contribution of the present paper is to show that the solution of \eqref{eq:SFPK} is unique under some mild conditions and hence $\{\Qx^N\}_{N\geq1}$ converges in $\mathcal{P}_2(\hat{S})$.
\end{remark}

The proof of Theorem \ref{thm:convergence-empirical} will be split into two steps: (i) we first verify the relative compactness of $\{\Qx_\mu^N\}_{N=1}^\infty\subset \mathcal{P}_2(S)$, where $\Qx_\mu^N:=Q_N\circ(\mu^N)^{-1}$; (ii) we then prove that for any limit of a convergent subsequence of $\{\Qx^N\}_{N=1}^\infty$, there exists an $\hat{S}$-valued r.v. $(\tilde{\mu}_0,\tilde{\theta},\tilde{\mu})$ defined on some probability space $(\widetilde{\Omega},\widetilde{\mathcal{F}},\widetilde{\Px})$ satisfying the stochastic FPK equation \eqref{eq:SFPK}.
%\begin{enumerate}
%  \item We first verify the relative compactness of $(\Qx_\mu^N)_{N=1}^\infty\subset \mathcal{P}_2(S)$, where
%      \[\Qx_\mu^N:=Q_N\circ(\mu^N)^{-1}.\]
%  \item We then prove that for any limit of a convergent subsequence of $(\Qx^N)_{N=1}^\infty$, there exists a $\mathcal{P}_2(E)\times\mathcal{L}^2\times S$-valued r.v. $(\tilde{\mu}_0,\tilde{\theta},\tilde{\mu})$ defined on some probability space $(\widetilde{\Omega},\widetilde{\mathcal{F}},\widetilde{\Px})$ satisfying the stochastic FPK equation \eqref{eq:SFPK}.
%\end{enumerate}
The next result, whose proof is reported in Appendix \ref{app:proof1}, completes the aformentioned first step.
\begin{lemma}\label{lem:conv-1}
 Under assumptions {\Asa} and {\Ath}, $\{\Qx_\mu^N\}_{N=1}^\infty$ is relatively compact in $\mathcal{P}_2(S)$.
\end{lemma}

By Prokhorov's theorem, $\{\Qx^N\}_{N\geq 1}$ is relatively compact, and we have the next result.

\begin{lemma}\label{lem:conv-2}
Under the assumptions in Theorem \ref{thm:convergence-empirical}, if the law of an $\hat{S}$-valued r.v. $(\tilde{\mu}_0,\tilde{\theta},\tilde{\mu})$, defined on some probability space $(\widetilde{\Omega},\widetilde{\mathcal{F}},\widetilde{\Px})$, is the limit of a convergent subsequence of $\{\Qx^N\}_{N\geq1}$, then $\tilde{\mu}$ satisfies the stochastic FPK equation \eqref{eq:SFPK}.
\end{lemma}

\begin{proof}
As $\widetilde{\Px}\circ(\tilde{\mu}_0,\tilde{\theta},\tilde{\mu})^{-1}$ is the weak limit of a convergent subsequence of $\{\Qx^N\}_{N\geq 1}$, still denoted by $\{\Qx^N\}_{N\geq 1}$, it follows from the Skorokhod representation theorem that there exist $\hat{S}$-valued random variables $(\bar{\mu}_0^N,\bar{\theta}^N,\bar{\mu}^N)$, $(\bar{\mu}_0,\bar{\theta},\bar{\mu})$ and the process $\overline{W}^0$ defined on some probability space $(\overline{\Omega},\overline{\mathcal{F}},\overline{\Px})$, such that $\overline{\Px}\circ(\bar{\mu}_0^N,\bar{\theta}^N,\bar{\mu}^N)^{-1}=\Qx^N$, $\overline{\Px}\circ(\bar{\mu}_0,\bar{\theta},\bar{\mu})^{-1} =\widetilde{\Px}\circ(\tilde{\mu}_0,\tilde{\theta},\tilde{\mu})^{-1}$, $\overline{W}^0$ is a Brownian motion under $(\overline{\Omega},\overline{\mathcal{F}},\overline{\Px})$, and
\begin{equation}\label{eq:mu^N-converge}
d_{\hat{S}}((\bar{\mu}_0^N,\bar{\theta}^N,\bar{\mu}^N),(\bar{\mu}_0,\bar{\theta},\bar{\mu}))\to 0,\quad N\to\infty,\ \widehat{\Px}\textrm{-a.s.}
\end{equation}
Let $\overline{\Ex}$ and $\widetilde{\Ex}$ be the expectation under probability measures $\overline{\Px}$ and $\widetilde{\Px}$ respectively. As in \eqref{eq:generator}, we consider the test function $\phi\in C_b^2(\R)$ in the rest of the proof. Also, let $C_\phi>0$ be a generic constant depending only on $\phi$ and the bound $K$ in the assumption {\Asa} that may change from place to place. First of all, we can deduce that
\begin{align*}
  &\lim_{N\to\infty}\overline{\Ex}\left[\sup_{t\in[0,T]}\left| \langle\bar{\mu}_t^N,\phi\rangle -\langle\bar{\mu}_0^N,\phi\rangle
  -\int_0^t \langle\bar{\mu}_s^N,\mathcal{A}^{\bar{\mu}_s^N,\bar{\theta}^N_s}\phi\rangle ds
  -\sigma_0\int_0^t \langle \bar{\mu}_s^N,\phi' \rangle d\overline{W}^0_s\right|^2\right]\\
  =&\lim_{N\to\infty}\Ex^{Q_N}\left[\sup_{t\in[0,T]}\left| \langle\mu_t^N,\phi\rangle -\langle\mu_0^N,\phi\rangle
  -\int_0^t \langle\mu_s^N,\mathcal{A}^{\mu_s^N,\widehat{\theta}^N_s}\phi\rangle ds
  -\sigma_0\int_0^t \langle \mu_s^N,\phi' \rangle d\widehat{W}^0_s\right|^2\right]\\
  =&\lim_{N\to\infty}\Ex^{Q_N}\left[\sup_{t\in[0,T]}\left| \frac{1}{N}\sum_{i=1}^N\sigma_i\int_0^t \langle \mu^N_s,\phi' \rangle\,d\widehat{W}^i_s \right|^2\right]\leq \lim_{N\to\infty}\frac{C_\phi}{N}=0.
\end{align*}
Moreover, by the definition in \eqref{eq:generator}, it holds that
\begin{equation}\label{eq:Upsilon}
\begin{aligned}
  &\left|\int_0^t \langle\bar{\mu}_s^N,\mathcal{A}^{\bar{\mu}_s^N,\bar{\theta}^N_s}\phi\rangle ds
  -\int_0^t \langle\bar{\mu}_s,\mathcal{A}^{\bar{\mu}_s,\bar{\theta}_s}\phi\rangle ds\right|\\
  \leq&\left|\int_0^t \langle\bar{\mu}_s^N,\mathcal{A}^{\bar{\mu}_s^N,\bar{\theta}^N_s}\phi\rangle ds
  -\int_0^t \langle\bar{\mu}_s^N,\mathcal{A}^{\bar{\mu}_s,\bar{\theta}_s}\phi\rangle ds\right| +\left|\int_0^t \langle\bar{\mu}_s^N,\mathcal{A}^{\bar{\mu}_s,\bar{\theta}_s}\phi\rangle ds
  -\int_0^t \langle\bar{\mu}_s,\mathcal{A}^{\bar{\mu}_s,\bar{\theta}_s}\phi\rangle ds\right|\\
  \leq&C_\phi\left(\|\bar{\theta}^N-\bar{\theta}\|_{\mathcal{L}_T^2}+d_S(\bar{\mu}^N,\bar{\mu})\right)+\left|\int_0^t \langle\bar{\mu}_s^N,\mathcal{A}^{\bar{\mu}_s,\bar{\theta}_s}\phi\rangle ds
  -\int_0^t \langle\bar{\mu}_s,\mathcal{A}^{\bar{\mu}_s,\bar{\theta}_s}\phi\rangle ds\right|.
\end{aligned}
\end{equation}
Recall \eqref{eq:generator} and $\bar{\mu}\in S$, we have $|\mathcal{A}^{\bar{\mu}_s,\bar{\theta}_s}\phi(x)|\leq C_\phi(1+|e|^2)$ for all $(s,e)\in[0,T]\times E$, where $e=(a,u,\sigma,y,x)$. Then, by \eqref{eq:mu^N-converge} and Theorem 7.12 of \cite{Villani03}, $\langle\bar{\mu}_s^N,\mathcal{A}^{\bar{\mu}_s,\bar{\theta}_s}\phi\rangle$ converges to $\langle\bar{\mu}_s,\mathcal{A}^{\bar{\mu}_s,\bar{\theta}_s}\phi\rangle$, as $N\to\infty$, $\overline{\Px}$-a.s.. Moreover, for $\ell(e):=|e|^2$, it holds that
\begin{align*}
\sup_{N\geq 1}\left|\langle\bar{\mu}_s^N,\mathcal{A}^{\bar{\mu}_s,\bar{\theta}_s}\phi\rangle\right|\leq C_\phi\left(1+\sup_{N\geq1}\langle\bar{\mu}_s^N,\ell\rangle\right)\leq C_\phi\left(1+\mathcal{W}_{E,2}(\bar{\mu}_s,\delta_0)+\sup_{N\geq1}d_S(\bar{\mu}^N,\bar{\mu})\right).
\end{align*}
The dominated convergence theorem implies that
\[\left|\int_0^t \langle\bar{\mu}_s^N,\mathcal{A}^{\bar{\mu}_s,\bar{\theta}_s}\phi\rangle ds
-\int_0^t \langle\bar{\mu}_s,\mathcal{A}^{\bar{\mu}_s,\bar{\theta}_s}\phi\rangle ds\right|\to 0,\quad N\to\infty,~~\overline{\Px}\textrm{-a.s.}.\]
and thus \eqref{eq:Upsilon} tends to 0 when $N\to\infty$, $\overline{\Px}\textrm{-a.s.}$.

For the stochastic integral term, it's easy to verify that $|\langle\bar{\mu}^N_s,\phi' \rangle-\langle \bar{\mu}_s,\phi' \rangle|^2\to 0$, $N\to\infty$,~ $\overline{\Px}$-a.s..
By the same reasoning as above and the Doob's maximal inequality, we have that
\[\bar{\Ex}\left[\sup_{t\in[0,T]}\left|\int_0^t (\langle \bar{\mu}^N_s,\phi' \rangle-\langle \bar{\mu}_s,\phi' \rangle)d\overline{W}^0_s\right|^2\right]\leq4\overline{\Ex}\left[\int_0^t \left| \langle \bar{\mu}^N_s,\phi' \rangle-\langle \bar{\mu}_s,\phi' \rangle \right|^2ds\right]\to 0,\]
when $N\to\infty$. It then follows from Fatou's lemma that
\begin{align*}
&~~~~\widetilde{\Ex}\left[\sup_{t\in[0,T]}\left| \langle\tilde{\mu}_t,\phi\rangle -\langle\tilde{\mu}_0,\phi\rangle
-\int_0^t \langle\tilde{\mu}_s,\mathcal{A}^{\tilde{\mu}_s,\tilde{\theta}_s}\phi\rangle ds
-\sigma_0\int_0^t \langle \tilde{\mu}_s,\phi' \rangle d\widetilde{W}^0_s\right|^2\right]\\
&=\overline{\Ex}\left[\sup_{t\in[0,T]}\left| \langle\bar{\mu}_t,\phi\rangle -\langle\bar{\mu}_0,\phi\rangle
-\int_0^t \langle\bar{\mu}_s,\mathcal{A}^{\bar{\mu}_s,\bar{\theta}_s}\phi\rangle ds
-\sigma_0\int_0^t \langle \bar{\mu}_s,\phi' \rangle d\overline{W}^0_s\right|^2\right]\\
&\leq\liminf_{N\to\infty}\overline{\Ex}\left[\sup_{t\in[0,T]}\left| \langle\bar{\mu}_t^N,\phi\rangle -\langle\bar{\mu}_0^N,\phi\rangle
-\int_0^t \langle\bar{\mu}_s^N,\mathcal{A}^{\bar{\mu}_s^N,\bar{\theta}^N_s}\phi\rangle ds
-\sigma_0\int_0^t \langle \bar{\mu}_s^N,\phi' \rangle d\overline{W}^0_s\right|^2\right]=0.
\end{align*}
As a result, \eqref{eq:SFPK} holds $\widetilde{\Px}$-a.s. as desired.
\end{proof}

The proof of the main result in this subsection readily follows.

\begin{proof}[Proof of Theorem \ref{thm:convergence-empirical}]
By Lemma \ref{lem:conv-2}, we have shown that any convergent subsequence of $\{\Qx^N\}_{N\geq1}$, still denoted by $\{\Qx^N\}_{N\geq1}$, converges weakly to some $\Qx\in\mathcal{P}(\hat{S})$. Note that $Q_N\circ(\mu_0^N)^{-1}$ for $N=1,2,\ldots$ are supported in a compact set and $\Theta\subset\R$ is also compact. By Lemma \ref{lem:conv-1}, $\Qx^N_\mu=Q_N\circ(\mu^N)^{-1}\in\mathcal{P}_2(S)$, and hence we have that
\begin{equation*}
  \lim_{R\to\infty}\sup_{N\geq1}\int_{\{\rho\in S:~d_S^2(\rho,\hat{\rho})\geq R)\}}d_S^2(\rho,\hat{\rho})\Qx_\mu^N(d\rho)=0.
\end{equation*}
Thanks to Theorem 7.12 of \cite{Villani03}, we have that $\mathcal{W}_{\hat{S},2}(\Qx^N,\Qx)\to0$ as $N\to\infty$.
\end{proof}

Inspired by \cite{KurtzXiong99}, we can further prove that the stochastic FPK \eqref{eq:SFPK} has a unique solution under some additional conditions, and thus  $\{\Qx^N\}_{N\geq 1}$ has a unique weak limit. Note that it is assumed that the coefficients of the FPK equation in \cite{KurtzXiong99} are deterministic, uniformly bounded and Lipschitz continuous, which is not satisfied in our framework. Some distinct arguments are required to conclude the uniqueness of the solution to \eqref{eq:SFPK}. For the reader's convenience, we recall the stochastic FPK equation \eqref{eq:SFPK} under the probability space $(\widetilde{\Omega},\widetilde{\mathcal{F}},\widetilde{\Px})$ in Theorem~\ref{thm:convergence-empirical} that
\begin{align*}%\label{eq:SFPK1}
\langle \tilde{\mu}_t,\phi \rangle=\langle \tilde{\mu}_0,\phi \rangle+\int_0^t \langle \tilde{\mu}_s,\mathcal{A}^{\tilde{\mu}_s,\tilde{\theta}_s}\phi \rangle ds+\sigma_0\int_0^t \langle \tilde{\mu}_s,\phi' \rangle d\widetilde{W}^0_s,\quad \forall~\phi\in C_b^2(\R),
\end{align*}
with the $\hat{S}$-valued r.v. $(\tilde{\mu}_0,\tilde{\theta},\tilde{\mu})$ and Brownian motion $\widetilde{W}^0$. We then introduce an auxiliary linear SPDE that, for a given $\nu\in S$,
\begin{equation}\label{eq:linear-equation}
\langle \vartheta_t,\phi \rangle=\langle \tilde{\mu}_0,\phi \rangle+\int_0^t \langle \vartheta_s,\mathcal{A}^{\nu_s,\tilde{\theta}_s}\phi \rangle ds+\sigma_0\int_0^t \langle \vartheta_s,\phi' \rangle d\widetilde{W}^0_s,\quad \forall~\phi\in C_b^2(\R),
\end{equation}
Using the methods similar to \cite{Kotelenez95} and \cite{KurtzXiong99}, we can obtain the next result, and its proof is reported in Appendix \ref{app:proof1}.
\begin{lemma}\label{lem:linear-wellpo}
If $\tilde{\mu}_0$ has a square-integrable density w.r.t. Lebesgue measure, the linear SPDE \eqref{eq:linear-equation} admits at most one solution.
\end{lemma}

With the aid of Lemma \ref{lem:linear-wellpo}, we can prove the desired result based on a conditional Mckean-Vlasov equation instead of a particle system proposed in \cite{KurtzXiong99}.
We have the next result on the uniqueness of solution to the stochastic FPK \eqref{eq:SFPK} under some additional assumptions.
\begin{proposition}\label{prop:conv-3}
Let assumptions {\Asa} and {\Ath} hold. If $\tilde{\mu}_0$ has a square-integrable density w.r.t. Lebesgue measure, then the stochastic FPK equation \eqref{eq:SFPK} has a unique solution.
\end{proposition}

\begin{proof}
Let $\tilde{\zeta}=(\tilde{a},\tilde{u},\tilde{\sigma},\widetilde{Y},X_0)$ be an $E$-valued r.v. such that $\widetilde{\Px}\circ\tilde{\zeta}^{-1}=\tilde{\mu}_0$. We first introduce the filtration $\widetilde{\mathcal{G}}_t:=\sigma(\tilde{\zeta})\vee\sigma(\widetilde{W}^0_s,\tilde{\theta}_s;~s\leq t)$ for $t\geq0$ and a Brownian motion $\widetilde{W}=(\widetilde{W}_t)_{t\in[0,T]}$ under $(\widetilde{\Omega},\widetilde{\mathcal{F}},\widetilde{\Px})$ with the natural filtration $\widetilde{\F}_t^W:=\sigma(\widetilde{W}_s;~s\leq t)$, which is assumed to be independent of $\widetilde{\zeta}$, $\widetilde{W}^0$ and $\tilde{\theta}$. We then define the following enlarged filtration given by
\begin{align}\label{eq:filtration}
\widetilde{\F}_t:=\widetilde{\mathcal{G}}_t\vee\widetilde{\F}_t^W=\sigma(\tilde{\zeta})\vee\sigma(\widetilde{W}_s,\widetilde{W}^0_s,\tilde{\theta}_s;~s\leq t),\quad t\geq0.
\end{align}
Then, the filtration $\widetilde{\mathbb{G}}=\{\widetilde{\mathcal{G}}_t\}_{t\in[0,T]}$ is immersed in the filtration $\widetilde{\mathbb{F}}=\{\widetilde{\mathcal{F}}_t\}_{t\in[0,T]}$, i.e., every bounded $\widetilde{\mathbb{G}}$-martingale is an $\widetilde{\mathbb{F}}$-martingale (by applying Theorem 3.2 in \cite{AksamitJeanblanc17}). We next consider the conditional Mckean-Vlasov equation that
\begin{equation}\label{eq:conditional-MV}
\left\{\begin{aligned}
dX_t&= \tilde{a}\left(\int_E x \,\nu_t(d\xi,dy,dx)-X_t\right)dt+\tilde{u}\tilde{\theta}_tdt+\tilde{\sigma} d\widetilde{W}_t+\sigma_0 d\widetilde{W}_t^0,\\
\nu_t&=\widetilde{\Px}_t\circ(\tilde{a},\tilde{u},\tilde{\sigma},\widetilde{Y},X_t)^{-1}.
\end{aligned}\right.
\end{equation}
Here, $\widetilde{\Px}_t\circ(\tilde{a},\tilde{u},\tilde{\sigma},\widetilde{Y},X_t)^{-1}$ denotes the law of $(\tilde{a},\tilde{u},\tilde{\sigma},\widetilde{Y},X_t)$ under $\widetilde{\Px}$ given $\widetilde{\mathcal{G}}_t$. In what follows, we will apply the contraction mapping theorem to show that \eqref{eq:conditional-MV} has a unique solution.

We use $\widetilde{\Ex}$ as the notation of the expectation under probability $\widetilde{\Px}$. Moreover, let us define, for $r>0$,
\begin{equation}\label{eq:Hc-space}
\mathbb{H}_r:=\left\{X:~\widetilde{\mathbb{F}}\textrm{-adapted process},~\|X\|_r:=\widetilde{\Ex}\left[\int_0^T e^{-r t}|X_t|dt\right]<\infty\right\}.
\end{equation}
It is straightforward to check that $(\mathbb{H}_r,\|\cdot\|_r)$ is a Banach space. For any $X^{(i)}\in\mathbb{H}_r$, $i=1,2$, we define the operator
\begin{equation}\label{eq:ContractionMap}
\mathcal{Z}_t(X^{(i)}):=X_0+\tilde{a}\int_0^t\left(\int_E x \,\nu_s^{(i)}(d\xi,dy,dx)-X_s^{(i)}\right)ds+\tilde{u}\int_0^t\tilde{\theta}_sds+\tilde{\sigma} \widetilde{W}_t+\sigma_0 \widetilde{W}_t^0,
\end{equation}
where $\nu_t^{(i)}=\widetilde{\Px}_t\circ(\tilde{a},\tilde{u},\tilde{\sigma},\widetilde{Y},X_t^{(i)})^{-1}$. Note that
\begin{align*}
\widetilde{\Ex}[|\mathcal{Z}_t(X^{(1)})-\mathcal{Z}_t(X^{(2)})|]
&\leq K\int_0^t \widetilde{\Ex}[|X^{(1)}_s-X^{(2)}_s|]+\widetilde{\Ex}\left[\int_E x(\nu^{(1)}_s-\nu^{(2)}_s)(d\xi,dy,dx)\right]ds\\
&\leq K\int_0^t \left\{\widetilde{\Ex}[|X^{(1)}_s-X^{(2)}_s|]+\widetilde{\Ex}\left[\mathcal{W}_{\R,1} (\widetilde{\Px}_s\circ(X_s^{(1)})^{-1}, \widetilde{\Px}_s\circ(X_s^{(2)})^{-1})\right]\right\}ds,
\end{align*}
where $K$ is the constant in the assumption {\Asa}. The second inequality is a direct consequence of the Kantorovich-Rubinstein duality formula (see, e.g. Theorem 5.10 in \cite{Villani09}). Moreover, by the Kantorovich duality, we have that
\begin{align*}
&\mathcal{W}_{\R,1}(\widetilde{\Px}_s\circ(X_s^{(1)})^{-1},\widetilde{\Px}_s\circ(X_s^{(2)})^{-1})\nonumber\\
&\qquad=\sup_{\phi\in\Lip_1(\R)}\left\{\int_\R\phi(x)\widetilde{\Px}_s\circ(X_s^{(1)})^{-1}(dx) -\int_\R\phi(y)\widetilde{\Px}_s\circ(X_s^{(2)})^{-1}(dy)\right\}\\
&\qquad=\sup_{\phi\in\Lip_1(\R)}\left\{\int_{\R\times\R}(\phi(x)-\phi(y))\widetilde{\Px}_s\circ(X_s^{(1)},X_s^{(2)})^{-1}(dx,dy) \right\}\\
&\qquad\leq \int_{\R\times\R}|x-y|\widetilde{\Px}_s\circ(X_s^{(1)},X_s^{(2)})^{-1}(dx,dy) \\
&\qquad=\widetilde{\Ex}[|X^{(1)}_s-X^{(2)}_s|\mid\widetilde{\mathcal{G}}_s],
\end{align*}
where $\Lip_1(\R)$ is the set of Lipschitz continuous functions on $\R$ with Lipschitz constants being less than one. Therefore, we have that
\begin{align*}
\widetilde{\Ex}[|\mathcal{Z}_t(X^{(1)})-\mathcal{Z}_t(X^{(2)})|]
\leq 2K\int_0^t \widetilde{\Ex}[|X^{(1)}_s-X^{(2)}_s|]ds.
\end{align*}
It follows from changing the order of integration that
\begin{align*}
\|\mathcal{Z}(X^{(1)})- \mathcal{Z}(X^{(2)})\|_r
&=\widetilde{\Ex}\left[\int_0^T e^{-r t}|\mathcal{Z}_t(X^{(1)})-\mathcal{Z}_t(X^{(2)})|dt\right]\\
&\leq 2K\widetilde{\Ex}\left[\int_0^T e^{-r t}\left(\int_0^t|X^{(1)}_s-X^{(2)}_s|ds\right)dt\right]\\
&\leq \frac{2K}{r}\|X^{(1)}-X^{(2)}\|_r.
\end{align*}
%{\color{blue}The last inequality only is a simple application of changing the order of integration, which has been used in the proof of Theorem \ref{thm:optimal-existence}. Specifically,
%\begin{align*}
%&~~~~\widetilde{\Ex}\left[\int_0^T e^{-r t}\left(\int_0^t|X^{(1)}_s-X^{(2)}_s|ds\right)dt\right]
%=\widetilde{\Ex}\left[\int_0^T |X^{(1)}_s-X^{(2)}_s|ds\int_s^T e^{-rt}dt\right]\\
%&=\frac{1}{r}\widetilde{\Ex}\left[\int_0^T (e^{-rs}-e^{-rT})|X^{(1)}_s-X^{(2)}_s|ds\right]
%\leq \frac{1}{r}\widetilde{\Ex}\left[\int_0^T e^{-rs}|X^{(1)}_s-X^{(2)}_s|ds\right]=\frac{1}{r}\|X^{(1)}-X^{(2)}\|_r.
%\end{align*}
%}
Thus, we choose $r>2K$, the operator in \eqref{eq:ContractionMap} is a strict contraction mapping and \eqref{eq:conditional-MV} has a unique solution.

Note that $\widetilde{\mathbb{G}}$ is immersed in $\widetilde{\mathbb{F}}$. It can be easily deduced from It\^{o}'s formula that if $(X,\nu)$ is the solution to \eqref{eq:conditional-MV}, then $\nu$ is a solution to \eqref{eq:SFPK}, i.e., for any $\phi\in C^2(\R)$ and $t\in [0,T]$, $\nu_t$ satisfies
\begin{align*}
\langle \nu_t,\phi \rangle=\langle \tilde{\mu}_0,\phi \rangle+\int_0^t \langle \nu_s,\mathcal{A}^{\nu_s,\tilde{\theta}_s}\phi \rangle ds+\sigma_0\int_0^t \langle \nu_s,\phi' \rangle d\widetilde{W}^0_s.
\end{align*}
Let $\nu^{(1)},\nu^{(2)}$ be two solutions of \eqref{eq:SFPK}, which are both $\widetilde{\mathbb{G}}$-adapted. Consider the following linear SDE with random coefficients
\begin{equation}\label{eq:linearSDE}
dX^{(1)}_t= \tilde{a}\left(\int_E x \,\nu^{(1)}_t(d\xi,dy,dx)-X^{(1)}_t\right)dt+\tilde{u}\tilde{\theta}_tdt+\tilde{\sigma} d\widetilde{W}_t+\sigma_0 d\widetilde{W}_t^0,\quad X^{(1)}_0=X_0,
\end{equation}
where $(\tilde{a},\tilde{u},\tilde{\sigma},\widetilde{Y},X_0)$ and $(\widetilde{W},\widetilde{W}^0)$ are defined the same as in \eqref{eq:conditional-MV}. Some standard arguments yield that \eqref{eq:linearSDE} has a unique $\widetilde{\mathbb{F}}$-adapted strong solution. Again, by It\^{o}'s formula and the immersion property, it holds that $\vartheta^{(1)}_t=\widetilde{\Px}_t\circ(\tilde{a},\tilde{u},\tilde{\sigma},\widetilde{Y},X^{(1)}_t)^{-1}$ satisfies that
\begin{align*}
\langle \vartheta^{(1)}_t,\phi \rangle=\langle \tilde{\mu}_0,\phi \rangle+\int_0^t \langle \vartheta^{(1)}_s,\mathcal{A}^{\nu^{(1)}_s,\tilde{\theta}_s}\phi \rangle ds+\sigma_0\int_0^t \langle \vartheta^{(1)}_s,\phi' \rangle d\widetilde{W}^0_s.
\end{align*}
However, it is assumed that $\nu^{(1)}$ is a solution to \eqref{eq:SFPK}, i.e.,
\begin{align*}
\langle \nu^{(1)}_t,\phi \rangle=\langle \tilde{\mu}_0,\phi \rangle+\int_0^t \langle \nu^{(1)}_s,\mathcal{A}^{\nu^{(1)}_s,\tilde{\theta}_s}\phi \rangle ds+\sigma_0\int_0^t \langle \nu^{(1)}_s,\phi' \rangle d\widetilde{W}^0_s.
\end{align*}
In view of Lemma \ref{lem:linear-wellpo}, we obtain that $\nu^{(1)}=\vartheta^{(1)}$. That is, $(X^{(1)},\nu^{(1)})$ is a solution to the conditional Mckean-Vlasov equation \eqref{eq:conditional-MV}. Similarly, $(X^{(2)},\nu^{(2)})$ is also a solution to \eqref{eq:conditional-MV}. The established uniqueness of the solution to \eqref{eq:conditional-MV} then yields the desired result that $\nu^{(1)}=\nu^{(2)}$.
\end{proof}

%{\color{red}The following remark should be deleted. Or we can say that if $\widetilde{W}$ is a Brownian motion independent of $\widetilde{\mathbb{G}}$, and $\widetilde{\mathbb{F}}$ is the nature filtration generated by $(\tilde{\zeta},\tilde{\theta},\widetilde{W}^0,\widetilde{W})$, then it can be easily deduced that $\widetilde{\mathbb{G}}$ is immersed in $\widetilde{\mathbb{F}}$.}
%{\color{blue}\begin{remark}
%If $\theta$ is adapted to the filtration generated by ${\bm \zeta}$ and $W^0$, i.e.
%\[\widetilde{\mathcal{G}}_t=\sigma(\tilde{\bm{\zeta}})\vee\sigma(W^0_s,\theta_s;s\leq t)=\sigma(\tilde{\bm{\zeta}})\vee\sigma(W^0_s;s\leq t),\]
%and in this case, we have that
%\[\widetilde{\mathcal{F}}_t:=\sigma(\tilde{\bm{\zeta}})\vee\sigma(W^0_s,\mathbf{W}_s;s\leq t).\]
%It can be easily deduced that $\widetilde{\mathbb{G}}$ is immersed in $\widetilde{\mathbb{F}}$. Some examples can be found in Section 1.1 of Volume II of \cite{CarmonaDelarue17} and Section 3.1 of \cite{AksamitJeanblanc17}.
%\end{remark}}

Proposition \ref{prop:conv-3} guarantees that the stochastic FPK equation \eqref{eq:SFPK} has a unique solution under some mild conditions. As a result, the relatively compact sequence $\{\Qx^N\}_{N\geq 1}$ in Theorem \ref{thm:convergence-empirical} has a unique weak limit under the same conditions.
\begin{corollary}\label{coro:conv-3.5}
Under assumptions in Proposition \ref{prop:conv-3}, the relatively compact sequence $\{\Qx^N\}_{N\geq 1}$ has a unique weak limit.
\end{corollary}

\subsection{Convergence of cost functionals}

We next proceed to show the convergence of objective functionals for a fixed $Q\in\mathcal{Q}(\nu)$. That is, we prove the convergence of $J_N^R(Q)$ defined in \eqref{eq:objectfuncRex} as $N\to\infty$.

Let $\widehat{\cal X}=(\widehat{\bm\zeta},(\widehat{W}^0,\widehat{\mathbf{W}}),\widehat{\theta})$ be the corresponding coordinate process to the weak control $Q$, and for all $i=1,2,\ldots$, $\widehat{X}^{\theta,i}=(\widehat{X}^{\theta,i}_t)_{t\in[0,T]}$ be the solution of \eqref{eq:systemmodel} driven by $\widehat{\cal X}$. Define a new empirical process by
\begin{equation}\label{eq:empirical-infty}
  \hat{\mu}^N_t:=\frac{1}{N}\sum_{i=1}^N\delta_{(\xi^i,\widehat{Y}^i,\widehat{X}^{\theta,i}_t)},\quad t\in[0,T].
\end{equation}
We can view $\hat{\mu}^N=(\hat{\mu}^N_t)_{t\in[0,T]}$ as the counterpart of $\mu^N$ in \eqref{eq:empirical-finite} that is driven by $(\widehat{\bm\zeta}^N,(\widehat{W}^{N,0},\widehat{\mathbf{W}}^N),\widehat{\theta}^N)$. The law of $\hat{\mu}^N$ under the weak control $Q$ is defined as:
\begin{align}\label{eq:hatQmu}
\widehat{\Qx}_\mu^N:=Q\circ(\hat{\mu}^N)^{-1}.
\end{align}
Recall that $L(x,y)=|x-y|^2$ is the quadratic cost function, the objective functional $J_N^R(Q)$ defined in \eqref{eq:objectfuncRex} can be reformulated by
\begin{equation}\label{eq:objectfuncFinal}
\begin{aligned}
J_N^R(Q)&=\alpha\Ex^Q[\langle \hat{\mu}^N_T,L \rangle] +\beta\Ex^Q\left[\int_0^T\langle\hat{\mu}^N_t,L\rangle dt\right]+\lambda\Ex^Q\left[\int_0^T|\widehat{\theta}_t|^2dt\right]\\
&=\alpha\int_S\langle\rho_T,L\rangle\widehat{\Qx}_\mu^N(d\rho) +\beta\int_0^T\left(\int_S\langle\rho_t,L\rangle\widehat{\Qx}_\mu^N(d\rho)\right)dt +\lambda\Ex^Q\left[\int_0^T|\widehat{\theta}_t|^2dt\right].
\end{aligned}
\end{equation}
To characterize the limiting objective function, we first impose the following assumption given by
\begin{itemize}
\item[{\Asb}] For any ${\bm\gamma}=(x^i,y^i)_{i\geq 1}\in\Xi^{\mathbb{N}}$, define $I_N:{\bm\gamma}\mapsto\frac{1}{N}\sum_{i=1}^N \delta_{(\xi^i,y^i,x^i)}\in\mathcal{P}_2(E)$ for $N\geq1$. Then, there exists a measurable mapping $I_{*}: \Xi^{\mathbb{N}}\to \mathcal{P}_{2}(E)$ such that
    \begin{equation}\label{eq:Istar}
    \nu\left(\left\{{\bm\gamma}\in\Xi^{\mathbb{N}}:~\lim_{N\to\infty} \mathcal{W}_{E,2}(I_N({\bm\gamma}),I_*({\bm\gamma}))=0\right\}\right)=1.
    \end{equation}
\end{itemize}
One simple example that the assumption {\Asb} holds is as follows. Let the sequence of initial money reserve and target levels  $(\zeta^i)_{i\geq1}:=(X_0^i,Y^i)_{i\geq1}$ be i.i.d.. Assume that the type vector $\xi^i$ given in \eqref{eq:parameters} converges to a deterministic $\xi\in\mathcal{O}$ as $i\to\infty$. Then, for $I_*:=\delta_{\xi}\otimes(\Px\circ(\zeta^1)^{-1})$, the LLN and the assumption {\Asa} yield that  $\mathbb{P}(\omega\in\Omega:\lim_{N_\to\infty}\mathcal{W}_{E,2}(I_N(\xi^i,\zeta^i(\omega)),I_*)=0)=1$. As a result, $\nu(\{{\bm\gamma}\in\Xi^{\mathbb{N}}:\lim_{N\to\infty} \mathcal{W}_{E,2}(I_N({\bm\gamma}),I_*)=0\})=\mathbb{P}(\omega\in\Omega:  \lim_{N_\to\infty}\mathcal{W}_{E,2}(I_N(\xi^i,\zeta^i(\omega)),I_*)=0)=1$.

It is not difficult to see that the assumption {\Asb} implies the validity of the condition \eqref{eq:convercondThm4-1} in Theorem~\ref{thm:convergence-empirical}. In fact, for any $Q\in\mathcal{Q}(\nu)$, it follows from \eqref{eq:Istar} that $Q\circ(\hat{\mu}^N_0,\widehat{\theta})^{-1}=Q\circ(I_N(\widehat{\bm\zeta}),\widehat{\theta})^{-1}\Rightarrow Q\circ (I_*(\widehat{\bm\zeta}),\widehat{\theta})^{-1}$, as $N\to\infty$. It is straightforward to conclude the next result.
\begin{lemma}\label{lem:MargConv-GivenQ}
Let assumptions {\Asa}, {\Asb} and {\Ath} hold. If $I_*(\widehat{\bm\zeta})$ has a square integrable density (under $Q$) w.r.t. Lebesgue measure, then the relatively compact sequence $\{\widehat{\Qx}_\mu^N\}_{N\geq 1}\subset\mathcal{P}_2(S)$ has a unique limit $\widehat{\Qx}_\mu\in\mathcal{P}_2(S)$ satisfying $\lim_{N\to\infty}\mathcal{W}_{S,2}(\widehat{\Qx}_\mu^N,\widehat{\Qx}_\mu)=0$.
\end{lemma}

Furthermore, for a given $Q\in\mathcal{Q}(\nu)$ and the unique limit $\widehat{\Qx}_\mu\in\mathcal{P}_2(S)$, let us define
\begin{equation}\label{eq:objectfuncFinalLimit}
  J^R(Q):=\alpha\int_S\langle\rho_T,L\rangle\widehat{\Qx}_\mu(d\rho) +\beta\int_0^T\left(\int_S\langle\rho_t,L\rangle\widehat{\Qx}_\mu(d\rho)\right)dt +\lambda\Ex^Q\left[\int_0^T|\widehat{\theta}_t|^2dt\right].
\end{equation}
Thanks to \eqref{eq:x-bound}, it can be deduced that $\sup_{N\geq 1} J_N^R(Q)<\infty$ for any $Q\in\mathcal{Q}(\nu)$. The next theorem is the main result of this section.

\begin{theorem}\label{thm:ObjecFuncConv}
Let assumptions in Lemma \ref{lem:MargConv-GivenQ} hold with $\nu\in{\cal P}(\Xi^{\mathbb{N}})$ satisfying $\sup_{i\in\N}\int_{\Xi^{\N}}(|x_i|^{2+\varrho}+|y_i|^2)\nu(d({\bf x},{\bf y}))<\infty$. Then, we have that
\begin{equation}\label{eq:ObjecFuncConv}
  \lim_{N\to\infty}J_N^R(Q)=J^R(Q),\quad \forall~Q\in\mathcal{Q}(\nu),
\end{equation}
where $J_N^R(Q)$ and $J^R(Q)$ are defined in \eqref{eq:objectfuncFinal} and \eqref{eq:objectfuncFinalLimit} respectively.
\end{theorem}

\begin{proof}
Recall that $S=C([0,T];\mathcal{P}_2(E))$, where $E$ consist of elements $e=(a,u,\sigma,y,x)\in\cO\times\Xi$. Similar to the proof of Lemma \ref{lem:conv-1}, denoting by $\delta_{0}\in S$ the constant process $\hat{\rho}_t\equiv\delta_0$, we have that
\[\langle\rho_t,L\rangle=\int_E L(x,y)\rho_t(de)\leq2\int_E|e|^2\rho_t(de) \leq2\sup_{t\in[0,T]}\mathcal{W}_{E,2}^2(\rho_t,\hat{\rho}_t)=2d_S^2(\rho,\hat{\rho}),\]
where $d_S$ is defined by \eqref{eq:metric-S}. That is, $\langle\rho_t,L\rangle$ has a quadratic growth as a function on $(S,d_S)$. We then claim that $\langle\rho_t,L\rangle$ is also continuous on $(S,d_S)$. Indeed, if $\{\rho^m\}_{m\geq1}$ and $\rho^\infty$ belong to $S$ and $d_S(\rho^m,\rho^\infty)\to0$ as $m\to\infty$, i.e., $\sup_{t\in[0,T]}\mathcal{W}_{E,2}(\rho_t^m,\rho_t^\infty)\to0$, as $m\to\infty$, then for the continuous function $L$ satisfying the quadratic growth condition, it holds from Theorem 7.12 of \cite{Villani03} that $\langle\rho_t^m,L\rangle\to\langle\rho_t^\infty,L\rangle$ as $m\to\infty$. Moreover, by Lemma \ref{lem:MargConv-GivenQ}, we have $\lim_{N\to\infty}\mathcal{W}_{S,2}(\widehat{\Qx}^N_\mu,\widehat{\Qx}_\mu)=0$. This yields from Theorem 7.12 of \cite{Villani03} again that
\begin{equation}\label{eq:hNConv}
  h_N(t):=\int_S\langle\rho_t,L\rangle\widehat{\Qx}_\mu^N(d\rho) \to\int_S\langle\rho_t,L\rangle\widehat{\Qx}_\mu(d\rho),\quad \forall~t\in[0,T],\quad N\to\infty.
\end{equation}
On the other hand, Lemma~\ref{lem:estimates} and Jensen's inequality imply the existence of a constant $C>0$ only depending on $T$ and $K$ that
\begin{align*}
  \sup_{N\geq1}\int_0^T|h_N(t)|^2dt&=\sup_{N\geq1}\int_0^T|\Ex^Q[\langle\hat{\mu}^N_t,L\rangle]|^2dt \leq\sup_{N\geq 1}\Ex^Q\left[\int_0^T|\langle\hat{\mu}^N_t,L\rangle|^2dt\right]\\
  &\leq\sup_{N\geq1}\frac{T}{N}\sum_{i=1}^N\Ex^Q\left[\sup_{t\in[0,T]}\left| \widehat{X}^{\widehat{\theta},i}_t-\widehat{Y}^i\right|^2\right]\leq C<\infty.
\end{align*}
It follows that $\{h_N(t)\}_{N\geq1}$ is uniformly integrable, and thus
\[\lim_{N\to\infty}\int_0^T h_N(t)dt= \lim_{N\to\infty}\int_0^T\left(\int_S\langle\rho_t,L\rangle\widehat{\Qx}_\mu^N(d\rho)\right)dt =\int_0^T\left(\int_S\langle\rho_t,L\rangle\widehat{\Qx}_\mu(d\rho)\right)dt.\]
From the result above and \eqref{eq:hNConv}, the desired convergence result follows.
\end{proof}

\section{Convergence of minimizers}\label{sec:gamma-conv}

Finally, this section is devoted to some convergence results of the optimal value functionals and the associated minimizers in the system with finite banks to the mean field model when $N\rightarrow\infty$. We first metrize the space $\mathcal{Q}(\nu)\subset\mathcal{P}_2(\Omega_\infty)$ by taking the quadratic Wasserstein distance $\mathcal{W}_{\Omega_\infty,2}$ on $\mathcal{Q}(\nu)$. The main result of the paper on the convergence of minimizers is stated as below.
\begin{theorem}\label{thm:MainThm}
Let assumptions {\Asa}, {\Asb} and {\Ath} hold. If $I_*({\bm\gamma})$ has a square-integrable density w.r.t. Lebesgue measure,  $\nu(d{\bm\gamma})$-a.e., then for any $\nu\in{\cal P}(\Xi^{\mathbb{N}})$ satisfying $\sup_{i\in\N}\int_{\Xi^{\N}}(|x_i|^{2+\varrho}+|y_i|^2)\nu(d({\bf x},{\bf y}))<\infty$, it holds that
\begin{equation}\label{eq:MinConv}
  \inf_{Q\in\mathcal{Q}(\nu)}J_N^R(Q)\to\inf_{Q\in\mathcal{Q}(\nu)}J^R(Q),\quad N\to\infty,
\end{equation}
where the infimum of $J^R(Q)$ can be attained by some $Q\in\mathcal{Q}(\nu)$. Furthermore, the minimizing sequence $\{Q_N\}_{N=1}^\infty\subset\mathcal{Q}(\nu)$ (up to a subsequence) converges to some $Q^*\in\mathcal{Q}(\nu)$ in $\mathcal{W}_{\Omega_\infty,2}$, then $Q^*$ is a minimizer of $J^R(Q)$ over $Q\in{\cal Q}(\nu)$.
\end{theorem}

To prove the above theorem, the key tool is the Gamma-convergence from $J_N^R$ to $J^R$ on $(\mathcal{Q}(\nu),\mathcal{W}_{\Omega_\infty,2})$ (see, e.g. \cite{Braides14}). We give its definition as below.

\begin{definition}
  $J^R:\mathcal{Q}(\nu)\to\R$ is the $\Gamma$-limit of $J_N^R:\mathcal{Q}(\nu)\to\R$, denoted by $J^R=\Gamma\textrm{-}\lim_{N\to\infty}J_N^R$, if the following two conditions hold:
  \begin{description}
    \item[(i) (liminf inequality):] for any $Q\in\mathcal{Q}(\nu)$ and any sequence $\{Q_N\}_{N=1}^\infty$ converging to $Q$ in $(\mathcal{Q}(\nu),\mathcal{W}_{\Omega_\infty,2})$, i.e. $\mathcal{W}_{\Omega_\infty,2}(Q_N,Q)\to0$ as $N\to\infty$, we have
        \[J^R(Q)\leq\liminf_{N\to\infty}J_N^R(Q_N);\]
    \item[(ii) (limsup inequality):] for any $Q\in\mathcal{Q}(\nu)$, there exists a sequence of $(\bar{Q}_N)_{N=1}^\infty$ converging to $Q$ in $(\mathcal{Q}(\nu),\mathcal{W}_{\Omega_\infty,2})$ such that
        \[J^R(Q)\geq\limsup_{N\to\infty}J_N^R(\bar{Q}_N).\]
  \end{description}
\end{definition}
%{\color{blue}The red part has been moved to the last subsection, it should be deleted here.}
%{\color{red}To establish the Gamma-convergence from $J_N^R$ to $J^R$, we impose the following assumption:
%\begin{itemize}
%\item[{\Asb}] For any $\widehat{\bm\zeta}=(\widehat{X}^i_0,\widehat{Y}^i)_{i\geq 1}\in\Xi_K^{\mathbb{N}}$, define $I_N:\widehat{\bm\zeta}\mapsto\frac{1}{N}\sum_{i=1}^N \delta_{(\xi^i,\widehat{Y}^i,\widehat{X}^i_0)}\in\mathcal{P}_2(E)$ for $N\geq1$. Then, there exists a measurable mapping $I_{*}: \Xi_{K}^{\mathbb{N}}\to \mathcal{P}_{2}(E)$ such that
%    \begin{equation}\label{eq:Istar-old}
%    \nu\left(\left\{\widehat{\bm\zeta}\in\Xi_K^{\mathbb{N}}:~\lim_{N\to\infty} \mathcal{W}_{E,2}(I_N(\widehat{\bm\zeta}),I_*(\widehat{\bm\zeta}))=0\right\}\right)=1.
%    \end{equation}
%\end{itemize}
%It is not difficult to see that the assumption {\Asb} implies the validity of the condition \eqref{eq:convercondThm4-1} in Theorem~\ref{thm:convergence-empirical}.}
The next result implies the Gamma-convergence, which serves as an important step before we prove Theorem \ref{thm:MainThm}.
\begin{proposition}\label{prop:gamma-conv}
Let assumptions in Theorem~\ref{thm:MainThm} hold. %Let assumptions {\Asa}, {\Asb} and {\Ath} hold.
For any $\{Q_N\}_{N\geq1},Q\subset{\cal Q}(\nu)$ satisfying $\lim_{N\to\infty}\mathcal{W}_{\Omega_\infty,2}(Q_N,Q)=0$, let $(\widehat{\bm\zeta}^N,(\widehat{W}^{N,0},\widehat{\mathbf{W}}^N),\widehat{\theta}^N)$ (resp. $(\widehat{\bm\zeta},(\widehat{W}^{0},\widehat{\mathbf{W}}),\widehat{\theta})$) be the corresponding coordinate process to $Q_N$ (resp. $Q$). If $I_*(\widehat{\bm\zeta})$ has a square-integrable density (under $Q$) w.r.t. Lebesgue measure, then we have
%\[J^R=\Gamma\textrm{-}\lim_{N\to\infty}J_N^R.\]
\[\lim_{N\to\infty}J_N^R(Q_N)=J^R(Q).\]
\end{proposition}

\begin{proof}
%Let $\{Q_N\}_{N\geq 1}$ and $Q$  be elements in $\mathcal{Q}(\nu)$ satisfying $\lim_{N\to\infty}\mathcal{W}_{\Omega_\infty,2}(Q_N,Q)=0$.
%{\color{red} Let $({\bm\zeta}^N,(W^{N,0},\mathbf{W}^N),\theta^N)$ (resp. $({\bm\zeta},(W^{0},\mathbf{W}),\theta)$) be the corresponding coordinate process of $Q_N$ (resp. $Q$). Then we define the empirical process $\mu^N=(\mu_t^N)_{t\in[0,T]}\in S$ (resp. $\mu\in S$) as in \eqref{eq:empirical-finite}.
%It follows from the assumption {\Asa} and $\lim_{N\to\infty}\mathcal{W}_{\Omega_\infty,2}(Q_N,Q)=0$, we have $Q_N\circ(\mu_0^N,\theta^N)^{-1}\Rightarrow Q\circ(\mu_0,\theta)^{-1}$ as $N\to\infty$?}
%It follows that there exists $\hat{S}$-valued r.v.s $(\mu_0^N,\theta^N,\mu^N)$ and $(\mu_0,\theta,\mu)$ such that $\Qx^N=Q_N\circ(\mu^N_0,\theta^N,\mu^N)^{-1}$ converges to $\Qx=Q\circ(\mu_0,\theta,\mu)^{-1}$ as $N\to\infty$. {\color{red}what is type of convergence? How to obtain this convergence!}
%Let $(\widehat{\bm\zeta}^N,(\widehat{W}^{N,0},\widehat{\mathbf{W}}^N),\widehat{\theta}^N)$ (resp. $(\widehat{\bm\zeta},(\widehat{W}^{0},\widehat{\mathbf{W}}),\widehat{\theta})$) be the corresponding coordinate process of $Q_N$ (resp. $Q$).
For any point ${\bm\gamma}=(x^i,y^i)_{i\geq 1}\in\Xi^{\mathbb{N}}$, let us define the mappings:
\begin{align*}
\hat{I}_N:{\bm\gamma}\mapsto ({\bm\gamma},I_N({\bm\gamma}))\in \Xi^{\mathbb{N}}\times\mathcal{P}_2(E),\quad \hat{I}_*:{\bm\gamma}\mapsto ({\bm\gamma},I_*({\bm\gamma}))\in \Xi^{\mathbb{N}}\times\mathcal{P}_2(E),
\end{align*}
where $I_N$ and $I_*$ are defined in the assumption {\Asb}. Note that by Definition \ref{def:relax-sol} on $\mathcal{Q}(\nu)$, we have $Q_N\circ \hat{I}_N(\widehat{\bm\zeta}^N)^{-1}=\{Q_N\circ (\widehat{\bm\zeta}^N)^{-1}\}\circ \hat{I}_N^{-1}=\nu\circ\hat{I}_N^{-1}$, and $Q\circ \hat{I}_*(\widehat{\bm\zeta})^{-1}=\nu\circ\hat{I}_*^{-1}$. Moreover, the assumption {\Asb} yileds that $\nu\circ\hat{I}_N^{-1}\Rightarrow\nu\circ\hat{I}_*^{-1}$ as $N\to\infty$. Therefore, it holds that
\begin{equation}\label{eq:QI-N}
  Q_N\circ \hat{I}_N(\widehat{\bm\zeta}^N)^{-1}=\nu\circ\hat{I}_N^{-1}~\Rightarrow~ \nu\circ\hat{I}_*^{-1}=Q\circ \hat{I}_*(\widehat{\bm\zeta})^{-1},\quad N\to\infty.
\end{equation}
As a  direct consequence of $\lim_{N\to\infty}\mathcal{W}_{\Omega_\infty,2}(Q_N,Q)=0$, we note that $Q_N\circ(\widehat{\bm\zeta}^N,\widehat{\theta}^N)^{-1}\Rightarrow Q\circ(\widehat{\bm\zeta},\widehat{\theta})^{-1}$ as $N\to\infty$. It follows that $\{Q_N\circ(\widehat{\bm\zeta}^N,I_N(\widehat{\bm\zeta}^N),\widehat{\theta}^N)^{-1}\}_{N\geq1}\subset \mathcal{P}(\Xi^{\mathbb{N}}\times\mathcal{P}_2(E)\times \mathcal{L}_T^2)$ is tight.

Next, we prove that $Q_N\circ(\widehat{\bm\zeta}^N,I_N(\widehat{\bm\zeta}^N),\widehat{\theta}^N)^{-1}$ converges weakly as $N\to\infty$. Without loss of generality, for $i=1,2$, assume the subsequence $Q_{N_k^i}\circ(\widehat{\bm\zeta}^{N_k^i},I_{N_k^i}(\widehat{\bm\zeta}^{N_k^i}),\widehat{\theta}^{N_k^i})^{-1}$ converges weakly to $\Px^i\circ({\bm\zeta}^i,J^i,\theta^i)^{-1}$ respectively, where $({\bm\zeta}^i,J^i,\theta^i)$ is $\Xi^{\mathbb{N}}\times\mathcal{P}_2(E)\times \mathcal{L}_T^2$-valued random variable defined on some probability space $(\Omega^i,\mathcal{F}^i,\Px^i)$ respectively. Note that $\Px^1\circ({\bm\zeta}^1,\theta^1)^{-1}=\Px^2\circ({\bm\zeta}^2,\theta^2)^{-1}$, then by the Gluing Lemma (see, e.g. Lemma 7.6 in \cite{Villani09}), there is a probability space $(\overline{\Omega},\overline{\mathcal{F}},\overline{\Px})$ on which $(\bar{\bm\zeta},\bar{J}_1,\bar{J}_2,\bar{\theta})$ is defined such that $\overline{\Px}\circ(\bar{\bm\zeta},\bar{J}_i,\bar{\theta})^{-1}=\Px^i\circ({\bm\zeta}^i,J^i,\theta^i)^{-1}$ for $i=1,2$. Moreover, it follows from \eqref{eq:QI-N} that $\Px^i\circ({\bm\zeta}^i,J^i)^{-1}=\nu\circ\hat{I}_*^{-1}$, and hence
\[\Px^{i}\left(\{\omega\in\Omega^{i}; J_{i}=I_*({\bm\zeta}^i(\omega))\}\right)=\nu\left(\left\{{\bm\gamma}\in\Xi^{\mathbb{N}}; \hat{I}_*({\bm\gamma})=({\bm\gamma},I_*({\bm\gamma}))\right\}\right)=1,\]
which implies that $\bar{J}_1=\bar{J}_2=I_*(\bar{\bm\zeta})$, $\overline{\Px}$-a.s. Therefore $(\bar{\bm\zeta},\bar{J}_1,\bar{\theta})=(\bar{\bm\zeta},\bar{J}_2,\bar{\theta})$, $\overline{\Px}$-a.s.. This yields that $\Px^1\circ({\bm\zeta}^1,J^1,\theta^1)^{-1}=\Px^2\circ({\bm\zeta}^2,J^2,\theta^2)^{-1}$. Thus, the limit is unique, i.e., $Q_N\circ(\widehat{\bm\zeta}^N,I_N(\widehat{\bm\zeta}^N),\widehat{\theta}^N)^{-1}\Rightarrow Q\circ(\widehat{\bm\zeta},I_*(\widehat{\bm\zeta}),\widehat{\theta})$ as $N\to\infty$.

Define the empirical process $\mu^N=(\mu_t^N)_{t\in[0,T]}\in S$ as in \eqref{eq:empirical-finite}, then we have $Q_N\circ(\mu_0^N,\widehat{\theta}^N)^{-1}\Rightarrow Q\circ(I_*(\widehat{\bm\zeta}),\widehat{\theta})^{-1}$ as $N\to\infty$. Since $I_*(\widehat{\bm\zeta})$ has a square-integrable density (under $Q$) w.r.t. Lebesgue measure, by Theorem~\ref{thm:convergence-empirical} and Corollary \ref{coro:conv-3.5}, $\Qx_\mu^N=Q_N\circ(\mu^N)^{-1}$ converges to $\Qx_\mu=Q\circ\mu^{-1}$ in $\mathcal{P}_2(S)$ as $N\to\infty$. Using similar arguments to prove Theorem \ref{thm:ObjecFuncConv}, it can be deduced that
\begin{align*}
  &\lim_{N\to\infty}\alpha\int_S\langle\rho_T,L\rangle\Qx_\mu^N(d\rho) +\beta\int_0^T\left(\int_S\langle\rho_t,L\rangle\Qx_\mu^N(d\rho)\right)dt\\
  &\qquad=\alpha\int_S\langle\rho_T,L\rangle\Qx_\mu(d\rho) +\beta\int_0^T\left(\int_S\langle\rho_t,L\rangle\Qx_\mu(d\rho)\right)dt.
\end{align*}
Moreover, it follows from $Q_N\circ(\widehat{\theta}^N)^{-1}\Rightarrow Q\circ\widehat{\theta}^{-1}$ as $N\to\infty$, and the compactness of $\Theta$ that %$\lim_{N\to\infty}\Ex^{Q_N}[\|\theta^N\|_{\mathcal{L}^2_T}^2] =\Ex^{Q}[\|\theta\|_{\mathcal{L}^2_T}^2]$.
\[\lim_{N\to\infty}\Ex^{Q_N}\left[\|\widehat{\theta}^N\|_{\mathcal{L}^2_T}^2\right] =\Ex^{Q}\left[\|\widehat{\theta}\|_{\mathcal{L}^2_T}^2\right].\]
Accordingly, we have
\[\lim_{N\to\infty}J_N^R(Q_N)=J^R(Q),\qquad \textrm{when}\quad \lim_{N\to\infty}\mathcal{W}_{\Omega_\infty,2}(Q_N,Q)=0.\]
It yields that $J^R$ is both a liminf bound and a limsup bound of $(J_N^R)_{N=1}^\infty$ and the desired Gamma-convergence holds.
\end{proof}

With the help of Proposition \ref{prop:gamma-conv}, we only need to show the convergence (along a subsequence) of the minimizers obtained in Proposition \ref{prop:ex-op-relaxed}. However, based on Proposition \ref{prop:ex-op-relaxed} for each $N$, it is difficult to check the relative compactness of the sequence of minimizers under weak formulation directly. Therefore, in what follows, we will provide an alternative way to construct an optimal control to problem \eqref{eq:objectfuncRex} in the weak sense using Ekeland's variational principle, which also relies on the equivalence result \eqref{equitwo} established in Proposition \ref{prop:ex-op-relaxed}. For notational convenience, we omit the dependence of coefficients on $N$ similar to subsection \ref{subsec:op-con}.

Recall that $\mathbf{X}^{\theta}_t$ is governed by the SDE \eqref{eq:finite-x} in a compact form.
%the compact form for the interbank system with fixed $N\geq 1$,
%\begin{equation*}
%d\mathbf{X}^{\theta}_t=\mathbf{b}(\mathbf{X}^{\theta}_t,\theta_t)dt + \Sigma d\mathbf{W}_t^0,\ \ t\in[0,T],
%\end{equation*}
%where $\mathbf{b}$ and $\Sigma$ are defined as in \eqref{eq:coeff-b} and \eqref{eq:volSigma}, respectively.
Let us define the Hamiltonian $\mathcal{T}:\R^N\times\Theta\times\R^N\times\R^{N\times(N+1)}\to\R$ that
\[\mathcal{T}(\mathbf{x},\theta,\mathbf{p},\mathbf{q})=\mathbf{b}(\mathbf{x},\theta)^{\top}\mathbf{p} +\mathrm{tr}(\Sigma^\top \mathbf{q})+R_N(\mathbf{x},\mathbf{y};\theta),\]
where $\mathbf{b}(\mathbf{x},\theta)=A\mathbf{x}+\theta \mathbf{u}$, $A$ is defined by \eqref{eq:coeff-b} and $\Sigma$ is given in \eqref{eq:volSigma}. Then, for each $\theta\in\mathbb{U}^{\Px,\Fx}$, we have that $(\mathbf{p}_t^\theta,\mathbf{q}_t^\theta)\in \R^N\times\R^{N\times(N+1)}$ satisfies the adjoint BSDE that
\begin{align}\label{eq:adjoint-SDE}
d\mathbf{p}^\theta_t&=-\nabla_{\mathbf{x}} \mathcal{T}(\mathbf{X}_t^\theta,\theta_t,\mathbf{p}^\theta_t,\mathbf{q}^\theta_t)dt +\mathbf{q}^\theta_td\mathbf{W}^0_t=-\left[ A^\top\mathbf{p}^\theta_t+\frac{2\beta}{N}(\mathbf{X}_t^\theta-\mathbf{Y}) \right]dt+\mathbf{q}^\theta_td\mathbf{W}^0_t,\ \ t\in[0,T),\nonumber\\
\mathbf{p}_T^\theta&=\nabla_{\mathbf{x}}L_N(\mathbf{X}^\theta_T,\mathbf{Y}) =\frac{2\alpha}{N}(\mathbf{X}_T^\theta-\mathbf{Y}),
\end{align}
where $\mathbf{X}^\theta=(\mathbf{X}_t^\theta)_{t\in[0,T]}$ is the state process given by \eqref{eq:finite-x}. The standard result (see, e.g. Theorem 6.2.1 in \cite{Pham09}) shows that the BSDE \eqref{eq:adjoint-SDE} has a unique solution $(\mathbf{p}_t^\theta,\mathbf{q}_t^\theta)_{t\in[0,T]}$ and there exists a constant $D_3>0$ depending only on $K$ and $\Theta$  that
\begin{equation}\label{eq:pq-bound}
\sup_{\theta\in\mathbb{U}^{\Px,\Fx}} \Ex\left[\sup_{t\in[0,T]}|\mathbf{p}^\theta_t|^2+\int_0^T|\mathbf{q}^\theta_t|^2dt\right]\leq D_3.
\end{equation}

Next, we introduce the G\^{a}teaux derivative of the cost functional, which is a direct consequence of Corollary 4.11 of \cite{Carmona16}.

\begin{lemma}\label{lem:G-derivative}
Let assumptions {\Asa} and {\Ath} hold. Denote $\theta^{\varepsilon}:=\theta+\varepsilon(\hat{\theta}-\theta)$, $\forall~\theta,\hat{\theta}\in \mathbb{U}^{\Px,\Fx}$ and $\forall~\varepsilon\in[0,1]$. The G\^{a}teaux derivative of the cost functional is given by
\begin{equation*}
\left.\frac{d}{d\varepsilon}J_N(\theta^{\varepsilon})\right|_{\varepsilon=0}=\lim_{\varepsilon\to 0}\frac{1}{\varepsilon}(J_N(\theta^{\varepsilon})-J_N(\theta)) = \mathbb{E}\left[\int_{0}^{T}\eta^\theta_t\big(\hat{\theta}_t-\theta_t\big)dt\right],
\end{equation*}
where
\begin{equation}\label{eq:eta}
\eta^\theta_t := \frac{d}{d\theta}\mathcal{T}(\mathbf{X}_t^\theta,\theta_t,\mathbf{p}^\theta_t,\mathbf{q}^\theta_t)= 2\lambda\theta_t+\sum_{i=1}^{N}u_i p^{\theta,i}_t,\quad t\in[0,T],
\end{equation}
and $(\mathbf{p}^\theta,\mathbf{q}^\theta)=(\mathbf{p}^\theta_t,\mathbf{q}^\theta_t)_{t\in[0,T]}$ is the unique solution to the adjoint BSDE \eqref{eq:adjoint-SDE}.
\end{lemma}

The next result provides the construction of an optimal weak control using a different approach from Proposition \ref{prop:ex-op-relaxed}.
\begin{proposition}\label{prop:optimal-existence}
Let $\nu\in{\cal P}(\Xi^{\mathbb{N}})$ with $\sup_{i\in\N}\int_{\Xi^{\N}}(|x_i|^{2+\varrho}+|y_i|^2)\nu(d({\bf x},{\bf y}))<\infty$. Under assumptions {\Asa} and {\Ath}, there exists an optimal control $Q^*\in\mathcal{Q}(\nu)$ to the weak control problem \eqref{eq:objectfuncRex}. Moreover, let $\widehat{\cal X}:=(\widehat{\bm\zeta},(\widehat{W}^0,\widehat{\mathbf{W}}),\widehat{\theta}^*)$ be its coordinate process on $\Omega_{\infty}$, and $(\widehat{\mathbf{p}}^*,\widehat{\mathbf{q}}^*)=(\widehat{\mathbf{p}}_t^*,\widehat{\mathbf{q}}_t^*)_{t\in[0,T]}$ be the solution to the adjoint BSDE \eqref{eq:adjoint-SDE} with $(\mathbf{X}^\theta,\mathbf{Y},\mathbf{W}^0,\theta)$ replaced by $\widehat{\cal X}$ on the canonical probability space $(\Omega_{\infty},\widehat{\F}_{\infty},\widehat{\Fx},Q^*)$. Then, the optimal open-loop control $\widehat{\theta}^*=(\widehat{\theta}_t^*)_{t\in[0,T]}$ to the problem \eqref{eq:control-problem} driven by $(\widehat{\bm\zeta},(\widehat{W}^0,\widehat{\mathbf{W}}))$ satisfies that
\begin{equation}\label{eq:theta-star}
\widehat{\theta}^*_t = \Pi_{\Theta}\left(-\frac{1}{2\lambda}\sum_{j=1}^{N}u_i\widehat{p}^{j,*}_t\right), \quad t\in[0,T],
\end{equation}
where $\Pi_{\Theta}$ is the projection on $\Theta$ (see \eqref{eq:optimalMar}).
\end{proposition}

In preparation for the proof of Proposition \ref{prop:optimal-existence}, we first establish a minimizing sequence by using Ekeland's variational principle. Similar to the proof of Theorem 2.6 in \cite{Barbu-R-Z18}, we characterize the minimizing sequence by the associated adjoint processes and the projection mapping.

\begin{lemma}\label{lem:ob-continuous}
Under the assumptions {\Asa} and {\Ath}, the objective functional $J_N(\theta):\mathbb{U}^{\Px,\Fx}\to\R$ is continuous with respect to the metric induced by the $\Hx^2$-norm.
\end{lemma}

The proof is delegated in Appendix \ref{app:proof1}. According to Ekeland's variational principle (see, e.g. Theorem 1.45 in \cite{BauschkeCombettes11}), there exists a sequence $\{\theta^k\}_{k\geq 1}\subset\mathbb{U}^{\Px,\Fx}$, s.t.
\begin{equation}\label{eq:theta-pre-k}
J_N(\theta^k)\leq J_N(\theta)+ \frac{1}{k}\|\theta^k-\theta\|_{\mathbb{H}^2}, \qquad\forall~\theta\in \mathbb{U}^{\Px,\Fx}.
\end{equation}
Moreover, we can characterise $\theta^k$ by the associated adjoint process in the next result. The proof of the next lemma is reported in Appendix \ref{app:proof1}.
\begin{lemma}\label{lem:Ekeland}
For the sequence $\{\theta^k\}_{k\geq1}$ given by Ekeland's variational principle that satisfies \eqref{eq:theta-pre-k}, let $\mathbf{X}^{k}$ be the controlled diffusion under $\theta^k$ and $(\mathbf{p}^k,\mathbf{q}^k)$ be the solution to the adjoint BSDE \eqref{eq:adjoint-SDE} with $\theta$ replaced by $\theta^k$. For $t\in[0,T]$, there exists $\chi^k\in \Hx^2$ with $\|\chi^k\|_{\Hx^2}\leq 1$, such that $\theta_t^k$ satisfies
\begin{equation*}
\theta^k_t+\frac{1}{2\lambda}\tilde{\varphi}^k_t=-\frac{1}{2\lambda}\sum_{i=1}^N u_ip^{k,i}_t-\frac{1}{2\lambda k}\chi^k_t,
\end{equation*}
where $\tilde{\varphi}^k_t\in N_\Theta(\theta^k_t)$ and $N_\Theta(\theta^k_t)$ stands for the normal cone to $\Theta$ at $\theta_t^k$. Therefore, $\theta^k=(\theta_t^k)_{t\in[0,T]}$ admits the representation that
\begin{equation}\label{eq:theta-k}
\theta^k_t=\Pi_\Theta\left( -\frac{1}{2\lambda}\sum_{i=1}^N u_ip^{k,i}_t-\frac{1}{2\lambda k}\chi^k_t \right),\quad t\in[0,T].
\end{equation}
\end{lemma}

With the representation \eqref{eq:theta-k} of the minimizing sequence, we can finally prove Propostion \ref{prop:optimal-existence}. In particular, we show that the sequence of laws of the minimizing sequence $\{\theta^k\}_{k\geq 1}$ is tight and its limiting process induces one optimal weak control to the problem \eqref{eq:objectfuncRex}. As a consequence, we can obtain more properties on the optimal weak control, which yield the relative compactness of the minimizers.

\begin{proof}[Proof of Proposition \ref{prop:optimal-existence}]
Let $\{\theta^k\}_{k\geq1}$ be the sequence derived from the Ekeland's variational principle satisfying \eqref{eq:theta-pre-k} with the representation \eqref{eq:theta-k} in Lemma~\ref{lem:Ekeland}. Recall that ${\bm\zeta}=(\zeta^{i})_{i=1}^{\infty}$ is a $\Xi^{\N}$-valued r.v. with the law $\Px\circ{\bm\zeta}^{-1}=\nu$, and $(W^{0},\mathbf{W})$ is a sequence of independent Brownian motions under the original probability space $(\Omega,\F,\Px)$. Let $Q^k:=\Px\circ({\bm\zeta},(W^0,{\bf W}),\theta^k)^{-1}$, which is a probability measure on $\Omega_\infty$, and $\widehat{\cal X}^k:=(\widehat{\bm\zeta}^k,(\widehat{W}^0,{\bf \widehat{W}})^k,\widehat{\theta}^k)$ be its coordinate process on $\Omega_{\infty}$, i.e., $\widehat{\cal X}^k(\omega)=\omega$ for all $\omega\in\Omega$. Define $\widehat{\Fx}^k=(\F^{\widehat{\cal X}^k}_t)_{t\in[0,T]}$ the natural filtration generated by $\widehat{\cal X}^k$. By the definition of $Q^k$, it is easy to check that $Q^k\in Q(\nu)$ because: (i) $Q^k\circ(\widehat{{\bm\zeta}}^k)^{-1}=\Px\circ({\bm\zeta})^{-1}=\nu$; (ii) $(\widehat{W}^0,\widehat{\bf W})^k$ is independent Wiener process on $(\Omega_{\infty},\widehat{\F}_{\infty}^k,\widehat{\Fx}^k,Q^k)$; (iii) $\widehat{\theta}^k\in\mathbb{U}^{Q^k,\widehat{\Fx}^k}$.

We first show that the sequence $\{Q^k\}_{k\geq 1}\subset Q(\nu)$ is tight. Note that $(\Omega_\infty,d)$ is Polish and the marginal distributions on $\Xi^{\N}\times\mathcal{C}_T\times\mathcal{C}_T^\mathbb{N}$ equal to $\Px\circ({\bm\zeta},(W^0,{\bf W}))^{-1}$, we only need to show $\{Q_k\circ(\widehat{\theta}^k)^{-1}\}_{k\geq1}=\{\Px\circ(\theta^k)^{-1}\}_{k\geq1}$ is tight. In light of Lemma A.2 of \cite{Barbu-R-Z18}, it's sufficient to prove the following two conditions:
\begin{equation}\label{eq:space-criterion}
  \lim_{R\to\infty}\limsup_{k\to\infty}Q^k\left(y\in\mathcal{L}^2_T:\int_0^T|y_t|^2dt>R \right)=0,
\end{equation}
and for any $\varepsilon>0$,
\begin{equation}\label{eq:time-criterion}
  \lim_{\delta\to0}\limsup_{k\to\infty}Q^k\left(y\in\mathcal{L}^2_T:\sup_{h\in(0,\delta]}\int_0^{T-h} |y_{t+h}-y_t|^2dt>\varepsilon \right)=0.
\end{equation}

In fact, the compact containment condition \eqref{eq:space-criterion} is clearly satisfied by the compactness of $\Theta$. To prove \eqref{eq:time-criterion}, by the Chebyshev's inequality, it suffices to show
\begin{equation}\label{eq:time-criterion-p1}
  \lim_{\delta\to0}\limsup_{k\to\infty}\Ex\left[ \sup_{h\in(0,\delta]}\int_{0}^{T-h}|\theta^k_{t+h}-\theta^k_t|^2dt\right]=0.
\end{equation}
In light of \eqref{eq:theta-k} and the Lipschitz continuity of the projection mapping $\Pi_\Theta$, we have
\begin{align}
  \Ex\left[ \sup_{h\in(0,\delta]}\int_{0}^{T-h}|\theta^k_{t+h}-\theta^k_t|^2dt\right]
  \leq& C\Ex\left[\sup_{h\in(0,\delta]}\int_{0}^{T-h}\sum_{i=1}^N|p^{k,i}_{t+h}-p^{k,i}_t|^2 +\frac{1}{k}|\chi^k_{t+h}-\chi^k_t|^2dt\right]\nonumber\\
  \leq& C\Ex\left[\sup_{h\in(0,\delta]}\int_{0}^{T-h}|\mathbf{p}^k_{t+h}-\mathbf{p}^k_t|^2dt\right] +\frac{C}{k},\label{eq:time-criterion-p2}
\end{align}
where $C>0$ is a generic positive constant independent of $k$ and $\delta$ that may change from place to place. Moreover, by the Doob's maximal inequality and Lemma \ref{lem:estimates}, we have
\begin{align*}
  \Ex\left[ \sup_{h\in(0,\delta]}|\mathbf{p}^k_{t+h}-\mathbf{p}^k_t|^2\right]&\leq C\Ex\left[ \sup_{h\in(0,\delta]}\int_t^{t+h}\left| A^\top\mathbf{p}^k_s+\frac{2\beta}{N}(\mathbf{X}_s^k-\mathbf{Y}) \right|^2ds\right]+C\Ex\left[ \sup_{h\in(0,\delta]}\left|\int_t^{t+h}\mathbf{q}_s^k d\mathbf{W}_s\right|^2\right]\\
  &\leq C\delta\Ex\left[\sup_{t\in[0,T]}|\mathbf{p}^k_t|^2+\sup_{t\in[0,T]}|\mathbf{X}^k_t|^2+K^2 \right]
  +C\Ex\left[ \int_t^{t+\delta}|\mathbf{q}_s^k|^2 ds\right]\\
  &\leq C\delta+C\Ex\left[ \int_t^{t+\delta}|\mathbf{q}_s^k|^2 ds\right],
\end{align*}
for some $C>0$. It then follows from the Fubini's theorem that
\begin{align*}
  \Ex\left[\sup_{h\in(0,\delta]}\int_{0}^{T-h}|\mathbf{p}^k_{t+h}-\mathbf{p}^k_t|^2dt\right]
  \leq& \Ex\left[\int_0^{T} \sup_{h\in(0,\delta]}|\mathbf{p}^k_{t+h}-\mathbf{p}^k_t|^2dt\right]\\
  \leq& C\delta+C\Ex\left[\int_0^{T}\int_t^{t+\delta}|\mathbf{q}_s^k|^2 dsdt\right]\\
  \leq& C\delta+C\Ex\left[\left(\int_0^\delta\int_0^s+\int_\delta^{T}\int_{s-\delta}^s +\int_{T}^{T+\delta}\int_{s-\delta}^{T-\delta}\right)dt|\mathbf{q}_s^k|^2 ds\right]\\
  \leq& C\delta+C\delta\Ex\left[\int_0^T|\mathbf{q}_s^k|^2 ds\right]\leq C\delta.
\end{align*}
Plugging this back into \eqref{eq:time-criterion-p2}, we get the desired result \eqref{eq:time-criterion-p1} and hence $\{Q^k\}_{k\geq 1}$ is tight.

Thanks to the Prokhorov's theorem, every subsequence $\{Q^k\}_{k\geq 1}$ admits a further subsequence converging weakly to some $Q^*\in\mathcal{Q}(\nu)$. For notational convenience, let us denote the convergent subsequence still by $\{Q^k\}_{k\geq 1}$ and it then follows that $Q^k\Rightarrow Q^*$. By virtue of the Skorokhod's representation theorem, there exists a probability space $(\widetilde{\Omega},\widetilde{\F},\widetilde{\Px})$, on which we can define the processes $(\widetilde{\bm\zeta}^k,(\widetilde{W}^0,\widetilde{\bf W})^k,\widetilde{\theta}^k)$ with $Q^k$, $k=1,2,\ldots$, and $(\widetilde{\bm\zeta}^*,(\widetilde{W}^0,\widetilde{\bf W})^*,\widetilde{\theta}^*)$ with $Q^*$ such that $\widetilde{\theta}^k\to\widetilde{\theta}^*$ a.s. Therefore, there exists $(\widehat{\bm\zeta},(\widehat{W}^0,\widehat{\bf W}),\widehat{\theta}^*)$ on $(\Omega_\infty,\mathcal{F}_\infty)$ satisfying $Q^*\circ(\widehat{\theta}^*)^{-1}=\widetilde{\Px}\circ(\widetilde{\theta}^*)^{-1}$. It then follows from Lemma \ref{lem:ob-continuous} that
\begin{equation*}
J_N^R(Q^*)=J_N(\widetilde{\theta}^*)=\lim_{k\to\infty} J_N(\widetilde{\theta}^k )=\lim_{k\to\infty}  J_N(\theta^k) = \inf_{\theta\in\mathbb{U}^{\Px,\Fx}}J_N(\theta).
\end{equation*}
Note that the representation \eqref{eq:theta-k} ensures that $\hat{\theta}^*$ is $\Theta$-valued, and hence $\|\hat{\theta}^*\|_{\mathbb{H}^2}<\infty$ and $\hat{\theta}^*$ is admissible. At last, by the equivalence result \eqref{equitwo} in Proposition \ref{prop:ex-op-relaxed}, we can conclude that
\begin{equation*}
J_N^R(Q^*)= \inf_{\theta\in\mathbb{U}^{\Px,\Fx}}J_N(\theta)=V_N^R(\nu)=\inf_{Q\in{\cal Q}(\nu)}J_N^R(Q),
\end{equation*}
which shows that $Q^*$ is indeed an optimal weak control.
\end{proof}

We are now in the position to prove the main result.

\begin{proof}[Proof of Theorem \ref{thm:MainThm}]
Let $Q_N$ be an optimal control in Proposition \ref{prop:optimal-existence} for each $N\geq 1$ and $(\widehat{\bm\zeta}^N,(\widehat{W}^{N,0},\widehat{\mathbf{W}}^N),\widehat{\theta}^N)$ be the corresponding coordinate process. With the help of \eqref{eq:pq-bound}, we can show that $Q_N\circ(\widehat{\theta}^N)^{-1}$ is tight similarly to the proof of Prop \ref{prop:optimal-existence}. It follows from Definition~\ref{def:relax-sol} of $\mathcal{Q}(\nu)$ that  $Q_N\circ(\widehat{\bm\zeta}^N)^{-1}=\Px\circ\widehat{\bm\zeta}^{-1}=\nu$ and $Q_N\circ(\widehat{W}^{N,0},\widehat{\mathbf{W}}^N)^{-1}=\Px\circ(\widehat{W}^0,\widehat{\mathbf{W}})^{-1}$. Then $\{Q_N\}_{N\geq 1}$ is tight if and only if $Q_N\circ(\widehat{\theta}^N)^{-1}$ is tight. Thus, it follows from the Prokhorov's theorem that $\{Q_N\}_{N\geq1}$ is relatively compact in $\mathcal{P}(\Omega_\infty)$. Moreover, in view of the definition of the metric defined in \eqref{eq:dinfty}, we have that $d_1$ and $d_2$ are uniformly bounded. On the other hand, $Q_N\circ(\widehat{W}^{N,0})^{-1}$ equals to the 1-dimensional Winer measure and $\Theta$ is compact. We can derive that
\begin{align*}
&\sup_{N\geq1}\int_{\Omega_\infty}d^2((\gamma,w,\varsigma,\kappa),(0,0,0,0))Q_N(d\gamma,dw,d\varsigma,d\kappa)\\
&\quad\leq \sup_{N\geq1}C\Ex^{Q_N}\left[ d_1^2(\widehat{\bm\zeta}_N,0)+d_2^2(\widehat{\mathbf{W}}^N,0)+d_3^2(\widehat{W}^{N,0},0)+d_4^2(\widehat{\theta}^N,0) \right]\\
&\quad=\sup_{N\geq1}C\Ex^{Q_N}\Bigg[\left(\sum_{i=1}^{\infty}2^{-i}\frac{|\widehat{\zeta}^{N,i}|}{1+|\widehat{\zeta}^{N,i}|}\right)^2
+\left(\sum_{i=1}^{\infty}2^{-i}\frac{\sup_{t\in[0,T]}|\widehat{W}_t^{N,i}|}{1+\sup_{t\in[0,T]}|\widehat{W}^{N,i}|}\right)^2\\
&\qquad\qquad\qquad\quad+\sup_{t\in[0,T]}|\widehat{W}_t^{N,0}|^2+\int_0^T|\widehat{\theta}_t^N|^2dt\Bigg]\\
&\quad\leq\sup_{N\geq1} C\Ex^{Q_N}\left[1+\sup_{t\in[0,T]}|\widehat{W}^{N,0}_t|^2+\|\widehat{\theta}^N\|_{\mathcal{L}^2_T}^2\right]<\infty,
\end{align*}
where $C>0$ is a generic positive constant independent of $N$. Therefore, again by Theorem 7.12 of \cite{Villani03}, it follows that $\{Q_N\}_{N\geq1}$ is relatively compact in $(\mathcal{P}_2(\Omega_\infty),\mathcal{W}_{\Omega_\infty,2})$. As a consequence, there exists $Q^*\in\mathcal{Q}(\nu)$ such that $\mathcal{W}_{\Omega_\infty,2}(Q_{N_k},Q^*)\to0$, where $\{Q_{N_k}\}_{N_k\geq1}$ is a convergent subsequence of $\{Q_N\}_{N\geq1}$. Let $(\widehat{\bm\zeta}^*,(\widehat{W}^{*,0},\widehat{\mathbf{W}}^*),\widehat{\theta}^*)$ be the corresponding coordinate process to $Q^*$. Note that
\begin{align}\label{eq:eqn000111}
Q^*\circ I_*(\widehat{\bm\zeta}^*)^{-1}=(Q^*\circ(\widehat{\bm\zeta}^*)^{-1})\circ I_*^{-1}=\nu\circ I_*^{-1}.
\end{align}
As $I_*({\bm\gamma})$ has a square-integrable density w.r.t. Lebesgue measure, $\nu(d{\bm\gamma})$-a.e., we have from  \eqref{eq:eqn000111} that $I_*(\widehat{\bm\zeta}^*)$ also has a square-integrable density (under $Q^*$) w.r.t. Lebesgue measure. Then, we can apply Proposition \ref{prop:gamma-conv} to conclude that $J_{N_k}^R(Q_{N_k})\to J^R(Q^*)$ as $k\to\infty$. Furthermore, from Theorem 2.1 of \cite{Braides14}, it follows that $Q^*$ is indeed a minimizer of $J^R(Q)$ over $Q\in{\cal Q}(\nu)$.
\end{proof}

\begin{remark}\label{rm47}
With the help of Theorem \ref{thm:ObjecFuncConv} and Thereom \ref{thm:MainThm}, we actually have established an approximate optimal weak control to the control problem \eqref{eq:control-problem} under strong formulation when $N$ is sufficiently large. More precisely, let $Q^*\in{\cal Q}(\nu)$ be the minimizer of $J^R(Q)$ over $Q\in{\cal Q}(\nu)$, i.e., $J^R(Q^*)=\inf_{Q\in\mathcal{Q}(\nu)}J^R(Q)$. Then, it follows from Proposition \ref{prop:ex-op-relaxed} that
\begin{align}\label{eq:approximJRNrem}
\lim_{N\to\infty}\left|J^R_N(Q^*)-\inf_{\theta\in\mathbb{U}^{\Px,\Fx}}J_N(\theta)\right|=0.
\end{align}
Indeed, we can see from Theorem~\ref{thm:MainThm} that
\begin{align*}
  &\left|J^R_N(Q^*)-\inf_{\theta\in\mathbb{U}^{\Px,\Fx}}J_N(\theta)\right|
  =\left|J^R_N(Q^*)-\inf_{Q\in{\cal Q}(\nu)}J_N^R(Q)\right|\\
  &\qquad\leq \left|J^R_N(Q^*)-J^R(Q^*)\right| +\left|\inf_{Q\in\mathcal{Q}(\nu)}J^R(Q)-\inf_{Q\in{\cal Q}(\nu)}J_N^R(Q)\right|\to 0,~~N\to\infty,
\end{align*}
where the last line converges to 0 thanks to \eqref{eq:ObjecFuncConv} and \eqref{eq:MinConv}.

It is worth noting that Theorem \ref{thm:convergence-empirical} only gives the weak convergence of $\{\mu^N\}_{N\geq1}$, and we can not use the weak limit to construct the mean field strong control problem directly. Moreover, it is difficult to establish the strong convergence result in $\Hx^2$ for the sequence of optimal strong controls based on Lemma \ref{lem:ex-op-strict} when $N\rightarrow\infty$. This is one main reason that we resort to the weak control formulation in the present paper.
\end{remark}

\begin{remark}
We also note here that the weak formulation and the Gamma-convergence arguments may also work for stochastic game problems in \cite{Carmona-F-S15} and \cite{Sun22} when all component banks are allowed to control their own lending and borrowing strategies. Let $Q_N\in{\cal Q}(\nu)$ be the Nash equilibrium of the $N$-player game in the weak formulation. Similar to Theorem $4.1$, one can possibly show that $\{Q_N\}_{N=1}^{\infty}$ (up to a subsequence) converges to some $Q^*\in{\cal Q}(\nu)$, where $Q^*$ is a mean filed equilibrium to the mean field game problem in the weak sense. Similar to Remark \ref{rm47}, it can be shown that $Q^*$ is an approximate Nash equilibrium for the finite player game when $N$ is sufficiently large.
\end{remark}

%Up to here, for the inter-bank system with $N$ banks, we have already established the existence of the optimal relaxed control in Proposition \ref{prop:ex-op-relaxed} that attains the value functions defined in both problem \eqref{eq:control-problem} and problem \eqref{eq:objectfuncRex}. However, in the next section, we will turn to study the relaxed control problem in a sizable interbank system when $N$ goes to infinity. To this end, some convergence results from $N$ banks model will be critical. However, based on Proposition \ref{prop:ex-op-relaxed} for each $N$, it is difficult to check the relative compactness of the sequence of relaxed controls directly. Therefore, in what follows, we will provide an alternative way to construct an optimal relaxed control to problem \eqref{eq:objectfuncRex} using Ekeland's variational principle, which also relies on the equivalence result \eqref{equitwo} established in Proposition \ref{prop:ex-op-relaxed}.

\section{Conclusions}
Motivated by the role of the central bank as the regulatory authority in stabilizing the interbank network, this paper studies a centralized monetary supply control in systems with $N$ banks and infinitely many banks by employing the weak control approach. First, in the system with finite heterogenous banks affected by a common noise, the existence of the optimal control under weak formulation is derived by the Ekeland's variational principle. As the number $N$ grows large, we further establish the convergence of optimizers from the model with $N$ banks by exercising the Gamma-convergence arguments in our framework and show that the limiting point corresponds to an optimal weak control in the mean field model.

Several future extensions can be conducted based on the current study. First, it is interesting to consider the monetary reserve dynamics of $N$ banks modelled by jump-diffusion processes similar to \cite{BoCapponi15} and generalize the  weak control approach and the Gamma-convergence arguments in the Skorokhod space. Moreover, the weak control formulation allows us to consider possible extensions of the interbank system to include clusters-hierarchy similar to \cite{Capponi-S-Y20} or multiple populations in the mean field model representing different types of banks. It is also appealing to examine strict minimum reserve requirements in the money supply control problem by setting the required reverses as dynamic floor constraints for all banks in a similar formulation of \cite{BoLiaoYu21}, for which we may consider some centralized singular control problems. Beyond the application in the interbank system, one may also modify the singular control problem in \cite{Yu21} as a centralized optimal dividend control by a sizable insurance group in models with $N$ and infinitely many subsidiaries when subsidiaries interact with others by default contagion and some fixed reinsurance rate. It will also be interesting to investigate the convergence error of the sequence of optimal (weak) controls (see, e.g. \cite{MottePham2021} that study the convergence error in the context of the mean-field Markov decision model).

\vspace{1cm}\noindent
{\small
\textbf{Acknowledgements}: The authors are grateful to two anonymous referees for their careful reading of the paper and constructive comments. L. Bo and T. Li are supported by Natural Science Foundation of China under grant no. 11971368 and National Center for Applied Mathematics in Shaanxi (NCAMS). X. Yu is supported by the Hong Kong Polytechnic University research grant under no. P0031417.}

\appendix

\section{Proofs of Auxiliary Results}\label{app:proof1}

\begin{proof}[Proof of Lemma \ref{lem:estimates}]
Let $C>0$ be a generic positive constant depending only on $K$ and $\Theta$, which may change from place to place. Recall the uniform boundedness of $a_i,\sigma_i$ and $\Theta$. For $\varrho>0$ in the assumption {\Asa}, we have that
\begin{equation}\label{eq:Xibound-p1}
\begin{aligned}
  |X_t^{\theta,i}|^{2+\varrho}&\leq C\Bigg\{ |X_0^i|^{2+\varrho}+|a_i|^{ 2+\varrho}\int_0^t\frac{1}{N}\sum_{j=1}^N |X^{\theta,j}_s|^{ 2+\varrho} ds+|a_i|^{ 2+\varrho}\int_0^t|X_s^{\theta,i}|^{2+\varrho}ds+|u_i|^{ 2+\varrho}\int_0^t|\theta_s|^{ 2+\varrho} ds\\
  &~~~~+|\sigma_i|^{ 2+\varrho} |W_t^i|^{ 2+\varrho}+|\sigma_0|^{ 2+\varrho} |W_t^0|^{ 2+\varrho}\Bigg\}\\
  &\leq C\left\{ 1+\int_0^t\frac{1}{N}\sum_{j=1}^N |X^{\theta,j}_s|^{ 2+\varrho} ds+\int_0^t|X_s^{\theta,i}|^{ 2+\varrho}ds+|W_t^i|^{ 2+\varrho}+|W_t^0|^{ 2+\varrho}\right\}.
\end{aligned}
\end{equation}
Summing the above inequality for $i$ from $1$ to $N$ and dividing by $N$ on both sides, we have that
\[\frac{1}{N}\sum_{i=1}^N|X_t^{\theta,i}|^{ 2+\varrho}\leq C\left\{ 1+\int_0^t\frac{2}{N}\sum_{j=1}^N |X^{\theta,j}_s|^{ 2+\varrho} ds+\frac{1}{N}\sum_{i=1}^N|W_t^i|^{ 2+\varrho}+|W_t^0|^{ 2+\varrho}\right\}.\]
Denoting $\overline{Z}_t^{\theta,N}:=\frac{1}{N}\sum_{i=1}^N|X_t^{\theta,i}|^{ 2+\varrho}$, we can deduce from the Doob's maximal inequality that
\[\begin{aligned}
  \Ex\left[\sup_{t\in[0,T]}\overline{Z}_t^{\theta,N}\right]&\leq C\left\{1 +\int_0^T\Ex\left[\sup_{s\in[0,t]}\overline{Z}_s^{\theta,N}\right]dt +\Ex\left[\sup_{t\in[0,T]}|W_t^0|^{ 2+\varrho}\right]\right\}\\
  &\leq C\left\{1 +\int_0^T\Ex\left[\sup_{s\in[0,t]}\overline{Z}_s^{\theta,N}\right]dt \right\}.
\end{aligned}\]
It follows from the Gronwall's inequality that
$\Ex\left[\sup_{t\in[0,T]}\overline{Z}_t^{\theta,N}\right]\leq Ce^{CT}<\infty$. Plugging the above estimate into \eqref{eq:Xibound-p1}, we obtain that
\[\begin{aligned}
  \Ex\left[\sup_{t\in[0,T]}|X_t^{\theta,i}|^{ 2+\varrho}\right]&\leq C\left\{1 +Ce^{CT}T+ \int_0^T\Ex\left[\sup_{s\in[0,t]}|X_s^{\theta,i}|^{ 2+\varrho}\right] +2\Ex\left[\sup_{t\in[0,T]}|W_t^0|^{ 2+\varrho}\right]\right\}\\
  &\leq C\left\{1+\int_0^T\Ex\left[\sup_{s\in[0,t]}|X_s^{\theta,i}|^{ 2+\varrho}\right] \right\}.
\end{aligned}\]
By Grownwall's inequality again, there exists a constant $D_1>0$, independent of $N,i$ and the choice of $\theta$, such that
\[\Ex\left[\sup_{t\in[0,T]}|X_t^{\theta,i}|^{ 2+\varrho}\right]\leq D_1<\infty.\]
Moreover, by the previous arguments and the BDG inequality, we have that
\begin{align*}
\Ex\left[ \sup_{|t-s|\leq\delta}\left|X^{\theta,i}_t-X^{\theta,i}_s\right|^{2} \right]
&\leq C\left\{D_1\delta+\Ex\left[ \sup_{|t-s|\leq\delta} |W_t^i-W_s^i|^{2}+\sup_{|t-s|\leq\delta}|W_t^0-W_s^0|^{2} \right]\right\}\\
&=C\left\{D_1\delta+2\Ex\left[ \sup_{|t-s|\leq\delta}\left|W_t^0-W_s^0\right|^{2} \right]\right\}.
\end{align*}
Note that $W^0=(W_t^0)_{t\in[0,T]}$ is a Brownian motion under $(\Omega,\F,\Px)$. Then, $\lim_{\delta\to0}\sup_{|t-s|\leq\delta}|W_t^0-W_s^0|^{2}=0$, $\Px$-a.s.. Moreover, by Doob's maximal inequality, we have $\Ex[\sup_{t\in[0,T]}|W_t^0|^2]\leq 4T$. Then, the second limit equality in \eqref{eq:x-bound} follows from the dominated convergence theorem. Thus, we complete the proof of the lemma.
\end{proof}

\begin{proof}[Proof of Lemma \ref{lem:wellpo-HJB}]
In view of the assumption $\prod_{i=1}^N\sigma_i\neq0$, the rank of $\Sigma$ is $N$. Moreover, the constant matrix $\Sigma\Sigma^\top$ is strictly positive, and thus there exists an invertible matrix $D$ such that $\Sigma\Sigma^\top=DD^\top$. As a result, for fixed $\mathbf{y}\in[-K,K]^N$, the parameterized Hamiltonian can be written by
\[\begin{aligned}\mathcal{H}(t,\mathbf{x},\mathbf{p},M;\mathbf{y}) &=\inf_{\theta\in\Theta}\left\{\mathbf{b}(\mathbf{x},\theta)^{\top}\mathbf{p} +\frac{1}{2}\mathrm{tr}(\Sigma\Sigma^\top M)+R_N(\mathbf{x},\mathbf{y};\theta)\right\}\\
&=\inf_{\theta\in\Theta}\left\{\mathbf{b}(\mathbf{x},\theta)^{\top}\mathbf{p} +\frac{1}{2}\mathrm{tr}(DD^\top M)+R_N(\mathbf{x},\mathbf{y};\theta)\right\}.\end{aligned}\]
We then have that
\begin{description}
  \item[(i)] The action space $\Theta\subset\R$ is compact.
  \item[(ii)] $\mathbf{b}(\mathbf{x},\theta)=A\mathbf{x}+D(D^{-1}\theta\mathbf{u}) =:\tilde{\mathbf{b}}(\mathbf{x})+D\kappa(\theta)$.
  \item[(iii)] $\tilde{\mathbf{b}}(\mathbf{x})\in C^2(\R^N)$ and $\nabla_{\mathbf{x}}\tilde{\mathbf{b}}$ is bounded; uniformly bounded $D$ and $\kappa$ are independent of $\mathbf{x}$;
  \item[(iv)] $R_N(\mathbf{x},\mathbf{y};\theta)$ and $L_N(\mathbf{x},\mathbf{y};\theta)$ are quadratic functions with respect to $\mathbf{x}$.
\end{description}
By Theorem VI.6.2 (Chapter VI) of \cite{FlemingRishel75}, the HJB equation \eqref{eq:HJB} has a unique classical solution with the parameter $\mathbf{y}\in\R^N$ satisfying the quadratic growth condition.
\end{proof}

\begin{proof}[Proof of Lemma \ref{lem:ex-op-strict}]
We first show that $\Ex[V(0,\mathbf{X}_0;\mathbf{Y})]\leq J_N(\theta)$, for any $\theta\in\mathbb{U}^{\Px,\Fx}$. Recall that $\mathbf{X}^{\theta}$ is the solution to \eqref{eq:finite-x} associated with $\theta$ for fixed $N$ on the original probability space $(\Omega,\F,\Px)$. For $\mathbf{y}\in \R^N$, let $V(t,\mathbf{x};\mathbf{y})$ be the unique classical solution of Eq.~\eqref{eq:HJB} and
\[\tau_m:=\inf\left\{t\in[0,T]:\int_0^t|\nabla_{\mathbf{x}}V(t,\mathbf{X}^{\theta}_s;\mathbf{Y})^\top\Sigma|^2ds >m\right\},\quad m>0,\]
with the convention that $\inf\varnothing=T$. It\^{o}'s formula gives that
\[\begin{aligned}
&V(\tau_m,\mathbf{X}^{\theta}_{\tau_m};\mathbf{Y})
=V(0,\mathbf{X}_0;\mathbf{Y})+\int_0^{\tau_m}\nabla_{\mathbf{x}}V(t,\mathbf{X}^{\theta}_t;\mathbf{Y})^\top\Sigma d\mathbf{W}_t^0\\
&\qquad\quad+\int_0^{\tau_m}\left\{\frac{\partial V}{\partial t}(t,\mathbf{X}^{\theta}_t;\mathbf{Y})+\mathbf{b}(\mathbf{X}^{\theta}_t,\theta_t)^\top \nabla_{\mathbf{x}}V(t,\mathbf{X}^{\theta}_t;\mathbf{Y})+\frac{1}{2}\mathrm{tr}(\Sigma\Sigma^\top \nabla_{\mathbf{x}}^2V(t,\mathbf{X}^{\theta}_t;\mathbf{Y}))\right\}dt.
\end{aligned}\]
Taking the expectation on both sides, we deduce from \eqref{eq:HJB} that
\[\begin{aligned}
\Ex[V(\tau_m,\mathbf{X}^{\theta}_{\tau_m};\mathbf{Y})]
&\geq\Ex[V(0,\mathbf{X}_0;\mathbf{Y})] -\Ex\left[\int_0^{\tau_m}R_N(\mathbf{X}^\theta_t,\mathbf{Y};\theta_t)dt\right]\\
&\geq\Ex[V(0,\mathbf{X}_0;\mathbf{Y})] -\Ex\left[\int_0^TR_N(\mathbf{X}^\theta_t,\mathbf{Y};\theta_t)dt\right],
\end{aligned}\]
where the last inequality holds as $R_N(\mathbf{x},\mathbf{y};\theta)$ defined in \eqref{eq:LNRN} is nonnegative. In view of the quadratic growth condition satisfied by the value function $V$ in Lemma~\ref{lem:wellpo-HJB}, we have that, $\Px$-a.s.,
\begin{align}\label{eq:boundVtaum00}
V(\tau_m,\mathbf{X}^{\theta}_{\tau_m};\mathbf{Y})\leq D_2\left(1+\sup_{t\in[0,T]}|\mathbf{X}_t^{\theta}|^2\right),\quad \forall~m>0,
\end{align}
where the constant $D_2>0$ which is independent of $m$ is given in Lemma~\ref{lem:wellpo-HJB}.
%\[\Ex[V(\tau_m,\mathbf{X}^{\theta}_{\tau_m};\mathbf{Y})]\leq D_3\Ex\left[1+\sup_{t\in[0,T]}|\mathbf{X}_t^{\theta}|^2\right],\quad m>0.\]
Moreover, it follows from Lemma \ref{lem:estimates} that $\Ex[\sup_{t\in[0,T]}|\mathbf{X}_t^{\theta}|^2]<+\infty$. Thus, in light of DCT, it holds that
\begin{align*}
\Ex[L_N(\mathbf{X}^{\theta}_T,\mathbf{Y})]&=\Ex[V(T,\mathbf{X}^{\theta}_T;\mathbf{Y})] =\lim_{m\to\infty}\Ex[V(\tau_m,\mathbf{X}^{\theta}_{\tau_m};\mathbf{Y})]\\
&\geq\Ex[V(0,\mathbf{X}_0;\mathbf{Y})] -\Ex\left[\int_0^TR_N(\mathbf{X}^\theta_t,\mathbf{Y};\theta_t)dt\right].
\end{align*}
This yields the inequality that
\begin{equation}\label{eq:V-leq}
  \Ex[V(0,\mathbf{X}_0;\mathbf{Y})]\leq J_N(\theta),\quad \forall~\theta\in\mathbb{U}^{\Px,\Fx}.
\end{equation}

On the other hand, it follows from \eqref{eq:optimalMar} and Lemma~\ref{lem:wellpo-HJB} that SDE \eqref{eq:Op-SDE} admits a unique strong solution (see Theorem 2.3.4 (Chapter 2) in \cite{Mao08}). Next, we show that $\Ex[V(0,\mathbf{X}_0;\mathbf{Y})]=J_N(\theta^*)$. Indeed, defining
\[\tau_m^*:=\inf\left\{t\in[0,T]:\int_0^t|\nabla_{\mathbf{x}}V(t,\mathbf{X}^*_s;\mathbf{Y})^\top\Sigma|^2ds >m\right\},\quad m>0,\]
we can deduce from \eqref{eq:optimalMar} that
\begin{align}\label{eq:taumstar001100}
&\Ex[V(\tau_m^*,\mathbf{X}^*_{\tau_m^*};\mathbf{Y})]=\Ex[V(0,\mathbf{X}_0;\mathbf{Y})]\nonumber\\
&\qquad\quad+\Ex\left[\int_0^{\tau_m^*}\bigg\{\frac{\partial V}{\partial t}(t,\mathbf{X}^*_t;\mathbf{Y})+\mathbf{b}(\mathbf{X}^*_t,\theta_t)^\top \nabla_{x}V(t,\mathbf{X}^*_t;\mathbf{Y})+\frac{1}{2}\mathrm{tr}(\Sigma\Sigma^\top \nabla_{x}^2V(t,\mathbf{X}^*_t;\mathbf{Y}))\bigg\}dt\right]\nonumber\\
&\quad=\Ex[V(0,\mathbf{X}_0;\mathbf{Y})] -\Ex\left[\int_0^{\tau_m^*}R_N(\mathbf{X}^*_t,\mathbf{Y};\theta_t^*)dt\right].
\end{align}
Similarly to \eqref{eq:boundVtaum00}, it follows from DCT that $\Ex[V(T,\mathbf{X}^*_{T};\mathbf{Y})]=\lim_{m\to\infty}\Ex[V(\tau_m^*,\mathbf{X}^*_{\tau_m^*};\mathbf{Y})]$. On the other hand, as $R_N(\mathbf{x},\mathbf{y};\theta)$ is nonnegative and satisfies the quadratic growth condition, it follows from Lemma \ref{lem:estimates} that, there exist constants $C_1,C_2>0$ independent of $m$ such that
\begin{align*}
\Ex\left[\int_0^{\tau_m^*}R_N(\mathbf{X}^*_t,\mathbf{Y};\theta_t^*)dt\right]\leq  C_1\Ex\left[1+\sup_{t\in[0,T]}|\mathbf{X}_t^{*}|^2\right]\leq C_2<+\infty.
\end{align*}
We also note that $R_N(\mathbf{x},\mathbf{y};\theta)$ is nonnegative and $m\to\tau_m^*$ is increasing. Then, the monotone convergence theorem (MCT) yields that
\begin{align*}
\Ex\left[\int_0^{T}R_N(\mathbf{X}^*_t,\mathbf{Y};\theta_t^*)dt\right]=\lim_{m\to\infty}\Ex\left[\int_0^{\tau_m^*}R_N(\mathbf{X}^*_t,\mathbf{Y};\theta_t^*)dt\right].
\end{align*}
Thus, the equality \eqref{eq:taumstar001100} gives that
\begin{align*}
\Ex[V(T,\mathbf{X}^*_T;\mathbf{Y})]=\Ex[V(0,\mathbf{X}_0;\mathbf{Y})] -\Ex\left[\int_0^TR_N(\mathbf{X}^*_t,\mathbf{Y};\theta_t^*)dt\right],
\end{align*}
which verifies that $\Ex[V(0,\mathbf{X}_0;\mathbf{Y})]=J_N(\theta^*)$. Combining with \eqref{eq:V-leq}, we have that
\begin{equation}\label{eq:strict-control}
  J_N(\theta^*)=\Ex[V(0,\mathbf{X}_0;\mathbf{Y})]=\inf_{\theta\in\mathbb{U}^{\Px,\Fx}}J_N(\theta).
\end{equation}
Thus, we complete the proof of the lemma.
\end{proof}

\begin{proof}[Proof of Lemma \ref{lem:conv-1}]
Similar to the arguments in \cite{Bo-C-L20}, we first check the tightness of $\{\Qx_\mu^N\}_{N=1}^\infty$. To this end, we show that for any $\varepsilon>0$, there exists a relatively compact set $\mathcal{M}\subset S=C([0,T];\mathcal{P}_2(E))$, such that
$\sup_{N\geq1} \Qx_\mu^N(\mathcal{M}^c)<\varepsilon$. Recall that a subset $\mathcal{M}\subset S$ is relatively compact iff the following two conditions hold: (i) the compact containment, i.e., for any $\rho\in S$ and $t\in[0,T]$, $\rho_t$ is contained in a relatively compact of $\mathcal{P}_2(E)$; (ii) the equicontinuous property, namely
\[\lim_{\delta\to0}\sup_{\rho\in S}\sup_{\substack{|t-s|\leq\delta,\\ 0\leq s,t\leq T}}\mathcal{W}_{E,2}(\rho_t,\rho_s)=0,\]
where $\mathcal{W}_{E,2}$ denotes the quadratic Wasserstein metric on $\mathcal{P}_2(E)$.

For any $M,\delta,\varepsilon>0$, we define
\[\mathcal{M}_1(M):=\left\{ \rho\in S;\sup_{t\in[0,T]}\int_{E} |e|^{ 2+\varrho}\rho_t(de)\leq M \right\},\quad\mathcal{M}_2(\delta,\varepsilon):=\left\{ \rho\in S;\sup_{|t-s|\leq\delta}\mathcal{W}_{E,2}(\rho_t,\rho_s)\leq \varepsilon \right\}.\]
We claim that $\mathcal{M}_1(M)\subset\mathcal{P}_2(E)$ and satisfies (i). Indeed, for any $\rho\in\mathcal{M}_1(M)$, it follows from Chebyshev's inequality that
\[\rho_t(B^c(R))\leq \frac{1}{R^{ 2+\varrho}}\sup_{t\in[0,T]}\int_E |e|^{2+\varrho}\rho_t(de)\leq\frac{M}{R^{ 2+\varrho}}\to0,\quad R\to\infty,\]
where $B(R)$ is the centered closed ball with radius $R$, i.e., $B(R)=\{e;|e|\leq R\}$. In addition, we get that
\[\sup_{\rho\in\mathcal{M}_1(M)}\int_{\{|e|^2\geq R\}}|e|^2\rho_t(de) \leq\frac{1}{R^{ \varrho/2}}\sup_{\rho\in\mathcal{M}_1(M)}\int_E |e|^{2+\varrho}\rho_t(de)\leq \frac{M}{R^{\varrho/2}}\to0,\quad R\to\infty.\]
Therefore, by Theorem 7.12 of \cite{Villani03}, $\mathcal{M}_1(M)$ is indeed relatively compact in $\mathcal{P}_2(E)$. In virtue of \eqref{eq:x-bound}, we have that
\begin{align*}
  \sup_{N\geq 1}\Qx_\mu^N(\mathcal{M}_1^c(M))
  =&\sup_{N\geq 1}Q_N(\mu^N\in\mathcal{M}_1^c(M))
  \leq\frac{1}{M}\sup_{N\geq 1}\Ex^{Q_N}\left[\sup_{t\in[0,T]}\int_{E} |e|^{ 2+\varrho}\,\mu^N_t(de)\right]\\
  \leq&\frac{1}{M}\sup_{N\geq 1}\frac{1}{N}\sum_{i=1}^N \Ex^{Q_N}\left[\sup_{t\in[0,T]}\left|\widehat{X}^{N,i}_t\right|^{ 2+\varrho}\right]
  \leq\frac{D_1}{M}\to0,\quad M\to\infty.
\end{align*}
That is, for any $\hat{\varepsilon}>0$, there exists $M_0>0$ such that $\sup_{N\geq 1}\Qx_\mu^N(\mathcal{M}_1^c(M_0))<\hat{\varepsilon}/2$. Moreover, by the compactness of $\Theta$ and the Doob's maximum inequality, we can deduce from \eqref{eq:x-bound} that
\begin{align*}
  \sup_{N\geq 1}\Qx_\mu^N(\mathcal{M}_2^c(\delta,\varepsilon))
  =&\sup_{N\geq 1}Q_N(\mu^N\in \mathcal{M}_2^c(\delta,\varepsilon))
  \leq\sup_{N\geq 1}\frac{1}{\varepsilon^2}\Ex^{Q_N}\left[ \sup_{|t-s|\leq\delta} \mathcal{W}_{E,2}^2(\mu^N_t,\mu^N_s) \right]\\
  \leq&\frac{1}{\varepsilon^2}\Ex^{Q_N}\left[ \sup_{|t-s|\leq\delta} \frac{1}{N}\sum_{i=1}^N|\widehat{X}^{N,i}_t-\widehat{X}^{N,i}_s|^2 \right]
  \to0,\quad\textrm{as }\delta\to 0.
\end{align*}
Therefore, for the same $\hat{\varepsilon}$ above, there exists a sequence $\{\delta_k\}_{k\geq 1}$ satisfying $\lim_{k\to\infty}\delta_k=0$ such that
\[\sup_{N\geq1}\Qx_\mu^N(\mathcal{M}_2^c(\delta_k,k^{-1}))<\frac{\hat{\varepsilon}}{2^k}.\]
Define $\mathcal{M}:=\mathcal{M}_1(M)\bigcap(\cap_{k\geq1}\mathcal{M}_2(\delta_k,k^{-1}))$, which is a relatively compact subset of $S$, and we have
$\sup_{N\geq1}\Qx_\mu^N(\mathcal{M}^c)<\hat{\varepsilon}$. It follows that $\{\Qx_\mu^N\}_{N\geq1}$ is tight. Furthermore, it can be deduced from the Prokhorov's theorem that $\{\Qx_\mu^N\}_{N\geq 1}$ is relatively compact in $\mathcal{P}(S)$.

To show $\{\Qx_\mu^N\}_{N\geq1}$ is relatively compact in $\mathcal{P}_2(S)$, we need to verify that $\{\Qx_\mu^N\}_{N\geq1}\subset\mathcal{P}_2(S)$. Recall the metric $d_S$ defined by \eqref{eq:metric-S} and define $\delta_0\in{\cal P}_2(E)$ as the constant process $\hat{\rho}=(\hat{\rho}_t)_{t\in[0,T]}\equiv\delta_0\in S$. Thanks to the assumption {\Asa} and \eqref{eq:x-bound}, we get that
\begin{align*}
  &\int_S d_S^2(\rho,\hat{\rho})\Qx_\mu^N(d\rho)=\Ex^{Q_N} \left[\sup_{t\in[0,T]}\mathcal{W}_{E,2}^2(\mu_t^N,\hat{\rho})\right]
  \leq\Ex^{Q_N}\left[\sup_{t\in[0,T]}\int_{E}|e|^2\mu_t^N(de)\right]\\
  &\qquad\leq\frac{1}{N}\sum_{i=1}^N\left\{\Ex^{Q_N}\left[|\xi_i|^2\right] +\Ex^{Q_N}\left[\sup_{t\in[0,T]}|\widehat{X}^{N,i}_t|^2\right]\right\}<\infty.
\end{align*}
Next, by Theorem 7.12 of \cite{Villani03}, it boils down to show the uniform integrability of $\{\Qx_\mu^N\}_{N\geq 1}$ that
\begin{equation}\label{eq:condition-UI}
  \lim_{R\to\infty}\sup_{N\geq1}\int_{\{\rho\in S:~d_S^2(\rho,\hat{\rho})\geq R)\}}d_S^2(\rho,\hat{\rho})\Qx_\mu^N(d\rho)=0.
\end{equation}
Actually, by \eqref{eq:x-bound}, it holds that
\[\sup_{N\geq1}\int_Sd_S^{2+\varrho}(\rho,\hat{\rho})\Qx_\mu^N(d\rho)\leq\sup_{N\geq1}\frac{1}{N} \sum_{i=1}^N\Ex^{Q_N}\left[\sup_{t\in[0,T]}|\widehat{X}^{N,i}_t|^{2+\varrho}\right]<\infty,\]
which yields the desired result \eqref{eq:condition-UI}, and the proof is complete.
\end{proof}

\begin{proof}[Proof of Lemma \ref{lem:linear-wellpo}]
Let us give some notations first. Let $\mathcal{M}(E)$ be the subset of all finite signed Borel measures on $E=\mathcal{O}\times\R^2$, whose first three marginal distributions equal to $\tilde{\mu}_0$ defined in Theorem \ref{thm:convergence-empirical}. That is, if $\vartheta\in\mathcal{M}(E)$, then $\vartheta(A_1\times A_2\times A_3\times A_4\times\R)=\tilde{\mu}_0(A_1\times A_2\times A_3\times A_4\times\R)$, $\forall A_1,A_2,A_3,A_4\in\mathcal{B}(\R)$. In order to show the uniqueness of the solution to \eqref{eq:linear-equation}, we adopt and modify some arguments in \cite{Kotelenez95} and \cite{KurtzXiong99}, in which the key idea is to transform the $\mathcal{M}(E)$-valued process to an $L^2(\R)$-valued process. To this end, for any measure $\vartheta\in\mathcal{M}(E)$ and the function $\psi\in L^2(\R)$, we respectively define
\begin{align*}
T_\delta \vartheta(x):=\int_E G_\delta(x-z)\vartheta(da,du,d\sigma,dy,dz),\quad T_\delta \psi(x):=\int_\R G_\delta(x-z)\psi(z)dz,\quad \delta>0,
\end{align*}
where $G_\delta(x):=\frac{1}{\sqrt{2\pi \delta}}e^{-\frac{x^2}{2\delta}}$ is the heat kernel. With a bit abuse of notations, we also denote
\[T_\delta (\sigma^2\vartheta)(x):=\int_{E} G_\delta(x-z)\sigma^2\vartheta(da,du,d\sigma,dy,dz).\]
Recall the definition of $\mathcal{A}^{m,\theta}$ in \eqref{eq:generator}, it follows from \eqref{eq:linear-equation} that
\begin{align*}
&\quad\langle T_\delta\vartheta_t,\psi \rangle_{L^2}=\langle \vartheta_t,T_\delta\psi \rangle\\
&=\langle \tilde{\mu}_0,T_\delta\psi \rangle+\int_0^t \langle \vartheta_s,g^{\nu_s,\tilde{\theta}}(\cdot)(T_\delta\psi)' \rangle ds+\int_0^t\left\langle \vartheta_s,\frac{\sigma^2+\sigma^2_0}{2}(T_\delta\psi)'' \right\rangle ds+\sigma_0\int_0^t \langle \vartheta_s,(T_\delta\psi)' \rangle d\widetilde{W}_s^0\\
&=\langle T_\delta\tilde{\mu}_0,\psi \rangle_{L^2}-\int_0^t \langle [T_\delta(g^{\nu_s,\tilde{\theta}}\vartheta_s)]',\psi \rangle_{L^2} ds+\int_0^t\left\langle \left[T_\delta\left(\frac{\sigma^2+\sigma_0^2}{2}\vartheta_s\right)\right]'',\psi \right\rangle_{L^2} ds
\\
&~~~~-\sigma_0\int_0^t \langle (T_\delta\vartheta_s)',\psi \rangle_{L^2} d\widetilde{W}_s^0,
\end{align*}
where $\langle f,g \rangle_{L^2}$ denotes the inner product in $L^2(\R)$, i.e., $\langle f,g \rangle_{L^2}=\int_\R f(x)g(x) dx$, and $\|f\|_{L^2}=\langle f,f \rangle_{L^2}^\frac{1}{2}$. It follows from It\^{o}'s lemma that
\begin{align*}
&~~~~\langle T_\delta\vartheta_t,\psi \rangle_{L^2}^2\\
&=\langle T_\delta\tilde{\mu}_0,\psi \rangle_{L^2}^2-2\int_0^t \langle T_\delta\vartheta_s,\psi \rangle_{L^2}\langle [T_\delta(g^{\nu_s,\tilde{\theta}}\vartheta_s)]',\psi \rangle_{L^2} ds+\int_0^t\langle T_\delta\vartheta_s,\psi \rangle_{L^2}\langle (T_\delta[(\sigma^2+\sigma_0^2)\vartheta_s])'',\psi \rangle_{L^2} ds\\
&\quad+\sigma_0^2 \int_0^t \langle (T_\delta\vartheta_s)',\psi \rangle_{L^2}^2 ds-2\sigma_0\int_0^t\langle T_\delta\vartheta_s,\psi \rangle_{L^2} \langle (T_\delta\vartheta_s)',\psi \rangle_{L^2} d\widetilde{W}^0_s.
\end{align*}
Summing over $\psi$ in a complete, orthonormal basis of $L^2(\R)$ and taking expectations on both sides, we have that
\begin{align*}
\widetilde{\Ex}\left[ \|T_\delta\vartheta_t\|_{L^2}^2\right]
&=\widetilde{\Ex}\left[ \|T_\delta\tilde{\mu}_0\|_{L^2}^2\right]-2\widetilde{\Ex}\left[\int_0^t \langle T_\delta\vartheta_s, [T_\delta(g^{\nu_s,\tilde{\theta}}\vartheta_s)]'\rangle_{L^2} ds\right]\\
&\quad+\widetilde{\Ex}\left[\int_0^t\langle T_\delta\vartheta_s,(T_\delta[(\sigma^2+\sigma_0^2)\vartheta_s])'' \rangle_{L^2} ds\right]
+\sigma_0^2\widetilde{\Ex}\left[ \int_0^t \| (T_\delta\vartheta_s)' \|_{L^2}^2 ds\right].
\end{align*}
Let $|\vartheta|$ be the total variation measure of $\vartheta\in\mathcal{M}(E)$. Similarly to the proof of Theorem 3.2 in \cite{KurtzXiong99}, it can be shown that there exists a constant $C>0$ such that
\begin{equation}\label{eq:estimate}
\widetilde{\Ex}\left[ \|T_\delta\vartheta_t\|_{L^2}^2\right]
\leq\widetilde{\Ex}\left[ \|T_\delta\tilde{\mu}_0\|_{L^2}^2\right]+C\int_0^t \widetilde{\Ex}\left[ \|T_\delta(|\vartheta_s|)\|_{L^2}^2 \right]ds.
\end{equation}
Note that if $\vartheta_t\in\mathcal{P}(E)$ and $(\vartheta_t)_{t\in[0,T]}$ satisfies \eqref{eq:linear-equation}, then $T_\delta(|\vartheta_t|)=T_\delta\vartheta_t$ for any $t\in[0,T]$.
%\[T_\delta(|\vartheta_t|)=T_\delta\vartheta_t.\]
Thus, by Grownwall's inequality, we have that
\begin{align}\label{eq:ineTdeltatheta00}
\widetilde{\Ex}\left[ \|T_\delta\vartheta_t\|_{L^2}^2\right]
\leq\widetilde{\Ex}\left[ \|T_\delta\tilde{\mu}_0\|_{L^2}^2\right]e^{CT}.
\end{align}

For any $\vartheta\in\mathcal{M}(E)$, let us use $\|\vartheta\|_{L^2}$ to denote its $L^2$-norm and say $\vartheta\in L^2$ if $\|\vartheta\|_{L^2}<\infty$. Recall that $\vartheta_0=\tilde{\mu}_0$ has an $L^2$-density with respect to Lebesgue measure. Then, for any complete and orthonormal basis $(\psi_j)_{j\geq1}$, we can deduce from the Fatou's lemma and \eqref{eq:ineTdeltatheta00} that any $\mathcal{P}(E)$-valued solution to \eqref{eq:linear-equation} satisfies that
\begin{align*}
\widetilde{\Ex}\left[\sum_{j=1}^\infty\langle\psi_j,\vartheta_t\rangle\right]
&=\widetilde{\Ex}\left[\sum_{j=1}^\infty\lim_{\delta\to0}\langle \psi_j,T_\delta\vartheta_t\rangle\right]
\leq \varliminf_{\delta\to0}\widetilde{\Ex}\left[\sum_{j=1}^\infty\langle \psi_j,T_\delta\vartheta_t\rangle\right]\\
&=\varliminf_{\delta\to0}\widetilde{\Ex}\left[\|T_\delta\vartheta_t\|_{L^2}\right]\leq\varliminf_{\delta\to0}\widetilde{\Ex}\left[ \|T_\delta\vartheta_0\|_{L^2}^2\right]e^{CT}.
\end{align*}
Note that $\|T_{\delta}\vartheta_0\|_{L^2}\leq\|\vartheta_0\|_{L^2}$ by applying Lemma 3.1 in \cite{KurtzXiong99}. It yields that, for all $t\in[0,T]$,
\begin{align*}
\widetilde{\Ex}\left[\sum_{j=1}^\infty\langle\psi_j,\vartheta_t\rangle\right]\leq\varliminf_{\delta\to0}\widetilde{\Ex}\left[ \|T_\delta\vartheta_0\|_{L^2}^2\right]e^{CT}\leq \|\vartheta_0\|_{L^2}^2e^{CT}<+\infty,
\end{align*}
i.e., $\vartheta_t\in L^2$ for any $t\in[0,T]$.

Let $\vartheta^1$ and $\vartheta^2$ be two $\mathcal{P}(E)$-valued solutions of \eqref{eq:linear-equation} with the same initial value $\tilde{\mu}_0$. Denote $\vartheta_t:=\vartheta^1_t-\vartheta^2_t$. It holds that
\[\widetilde{\Ex}\left[ \|T_\delta\vartheta_t\|_{L^2}^2\right]\leq C\int_0^t \widetilde{\Ex}\left[ \|T_\delta(|\vartheta_s|)\|_{L^2}^2 \right]ds.\]
Then, for a complete and orthonormal basis $(\psi_j)_{j\geq1}$, we can derive from Fatou's lemma again that
\begin{align*}
\widetilde{\Ex}\left[ \|\vartheta_t\|_{L^2}^2\right]
&=\widetilde{\Ex}\left[ \sum_j \langle \psi_j,\vartheta_t\rangle^2\right]
=\widetilde{\Ex}\left[ \sum_j \lim_{\delta\to 0}\langle T_\delta\psi_j,\vartheta_t\rangle^2\right]
=\widetilde{\Ex}\left[ \sum_j \lim_{\delta\to 0}\langle T_\delta\vartheta_t,\psi_j \rangle_{L^2}^2\right]\\
&\leq\varliminf_{\delta\to 0}\widetilde{\Ex}\left[ \|T_\delta\vartheta_t\|_{L^2}^2\right]
\leq \varliminf_{\delta\to 0}C\int_0^t \widetilde{\Ex}\left[ \|T_\delta(|\vartheta_s|)\|_{L^2}^2 \right]ds\\
&\leq C\int_0^t \widetilde{\Ex}\left[ \|\,|\vartheta_s|\,\|_{L^2}^2 \right]ds
=C\int_0^t \widetilde{\Ex}\left[ \|\vartheta_s\|_{L^2}^2 \right]ds,
\end{align*}
where the constant $C>0$ is the same to the one in \eqref{eq:estimate}. Accordingly, it follows from the Gronwall's inequality that \eqref{eq:linear-equation} has a unique solution.
\end{proof}

\begin{proof}[Proof of Lemma \ref{lem:ob-continuous}]
For $\theta^{1},~\theta^{2}\in\mathbb{U}^{\Px,\Fx}$, denote by $\mathbf{X}^{1},~\mathbf{X}^{2}$ the associated solutions to \eqref{eq:finite-x} respectively. Let $C>0$ be a generic positive constant depending only on $K$ and $\Theta$ that may change from place to place. It holds that
\begin{align*}
\Ex\left[\sup_{s\in[0,t]}|\mathbf{X}^{1}_s-\mathbf{X}^{2}_s|^2\right]
&\leq C\left(\|\theta^1-\theta^2\|_{\Hx^2}+ \Ex\left[\int_0^t\sup_{u\in[0,s]}|\mathbf{X}^{1}_u-\mathbf{X}^{2}_u|^2ds\right]\right)\leq C\|\theta^1-\theta^2\|_{\Hx^2},
\end{align*}
where the last inequality is due to the Grownwall's inequality. Recalling the definition of $J_N$ given in \eqref{eq:objectfunc}, we can deduce from Lemma \ref{lem:estimates} that
\begin{align*}
|J_N(\theta^1)-J_N(\theta^2)|
\leq C\left(\|\theta^1-\theta^2\|_{\Hx^2}+\Ex\left[\int_0^T |\mathbf{X}^{1}_t-\mathbf{X}^{2}_t|^2dt+|\mathbf{X}^{1}_T-\mathbf{X}^{2}_T|^2\right]\right)
\leq C\|\theta^1-\theta^2\|_{\Hx^2}.
\end{align*}
This shows that the objective functional $J_N$ is continuous on $\mathbb{U}^{\Px,\Fx}$ with respect to $\Hx^2$ norm.
\end{proof}

\begin{proof}[Proof of Lemma~\ref{lem:Ekeland}]
We first transform the constrained optimization problem into an unconstrained one.  Let $\idc_{\mathbb{U}^{\Px,\Fx}}(\theta):\mathbb{H}^2\to [0,+\infty]$ be the indicator function of $\mathbb{U}^{\Px,\Fx}$, namely
\begin{equation*}
  \idc_{\mathbb{U}^{\Px,\Fx}}(\theta):=
  \left\{\begin{array}{cl}
    0,&\text{if}\ \theta\in\mathbb{U}^{\Px,\Fx};\\
    +\infty,&\text{otherwise}.
  \end{array}\right.
\end{equation*}
Define $\widehat{J}_{N,k}:\mathbb{H}^2\to [0,+\infty]$ by
\begin{equation}\label{eq:J-N-hat}
  \widehat{J}_{N,k}(\theta) := J_N(\theta) + \frac{1}{k}\|\theta^k-\theta\|_{\mathbb{H}^2}+ \idc_{\mathbb{U}^{\Px,\Fx}}(\theta).
\end{equation}
Then, for each $k\geq 1$, $\theta^k$ is the global minimizer of $\widehat{J}_{N,k}$, hence it is a substationary point of $\widehat{J}_{N,k}(\theta)$. Accordingly, $\partial\hat{J}_{N,k}(\theta^k)$, the subdifferential of $\widehat{J}_{N,k}(\theta)$ at $\theta^k$, includes zero element, i.e. $0\in \partial\hat{J}_{N,k}(\theta^k)$. Note that this condition is crucial to formulate $\theta^k$. For this purpose, we first claim that
\begin{equation}\label{eq:substationary}
  \partial\widehat{J}_{N,k}(\theta)\subset \eta^\theta+\mathcal{N}_{\mathbb{U}^{\Px,\Fx}}(\theta)+ \frac{1}{k}\partial\|\theta^k-\theta\|_{\Hx^2},
\end{equation}
where $\eta^\theta$ is defined as \eqref{eq:eta} and $\mathcal{N}_{\mathbb{U}^{\Px,\Fx}}(\theta)$ is the normal cone to $\mathbb{U}^{\Px,\Fx}$ at $\theta$, i.e.
\begin{equation}\label{eq:normal-cone-H}
  \mathcal{N}_{\mathbb{U}^{\Px,\Fx}}(\theta)=\{y\in\Hx^2:\langle y,\hat{\theta}-\theta\rangle\leq0,\ \forall~\hat{\theta}\in\mathbb{U}^{\Px,\Fx}\}.
\end{equation}

Define $\widetilde{J}_N(\theta):=J_N(\theta)+\idc_{\mathbb{U}^{\Px,\Fx}}(\theta)$. For $\theta\in \mathbb{U}^{\Px,\Fx}$, we have
\begin{align*}
  \partial\widetilde{J}_N(\theta) = \{z\in\Hx^2: \widetilde{J}_N^{\uparrow}(\theta;y)\geq\langle y,z\rangle, \ \forall~ y\in\Hx^2\},
\end{align*}
where  $\widetilde{J}_N^{\uparrow}(\theta,y)$ is the upper subderivative at $\theta$ with respect to $y$:
{\small
\begin{align*}
  \widetilde{J}_N^{\uparrow}(\theta,y) = \sup_{V\subset\mathcal{N}(y)}\left[\limsup_{\substack {\theta'\to\theta,\alpha'\to\widetilde{J}_N(\theta)\\ \alpha'\geq \widetilde{J}_N(\theta'),t\to 0 }} \inf_{y' \in V}\bigg( \frac{J_N(\theta' + ty')-\alpha' + \idc_{\mathbb{U}^{\Px,\Fx}}(\theta' + ty')}{t}\bigg) \right],
\end{align*}
}
and $\mathcal{N}(y)$ is the set of all neighborhood of $y$. It follows that
\begin{align*}
  \widetilde{J}_N^{\uparrow}(\theta;y) = \lim_{t\rightarrow 0}\frac{J(\theta+ty)-J(\theta)}{t} + \idc^{'}_{\mathbb{U}^{\Px,\Fx}}(\theta,y),
\end{align*}
where
\begin{align*}
  \idc^{'}_{\mathbb{U}^{\Px,\Fx}}(\theta,y)=\left\{
  \begin{array}{cl}
    0,&\textrm{if}\ y\in T_{\mathbb{U}^{\Px,\Fx}}(\theta);\\
    \infty,&\textrm{if}\ y\notin T_{\mathbb{U}^{\Px,\Fx}}(\theta),
  \end{array}\right.
\end{align*}
and $T_{\mathbb{U}^{\Px,\Fx}}(\theta)$ is the (Clarke) tangent cone to $\mathbb{U}^{\Px,\Fx}$ at $\theta$ (see \cite{Rockafellar79}). Recall the definition of $\partial\widetilde{J}_N(\theta)$ and Lemma \ref{lem:G-derivative}, for any $z\in \partial\widetilde{J}_N(\theta)$, we have $\langle \eta^\theta,y \rangle\geq\langle y,z \rangle,$ for $y=\hat{\theta}-\theta$, $\hat{\theta}\in\mathbb{U}^{\Px,\Fx}$. It follows from \eqref{eq:normal-cone-H} that $z\in\eta^\theta+\mathcal{N}_{\mathbb{U}^{\Px,\Fx}}(\theta)$, i.e. $\partial\widetilde{J}_N(\theta)\subset\eta^\theta+\mathcal{N}_{\mathbb{U}^{\Px,\Fx}}(\theta)$. Moreover, thanks to Theorem 2 of \cite{Rockafellar79}, we have that
\begin{align*}
  \partial\widehat{J}_{N,k}(\theta)=\partial\left(\widetilde{J}_N(\theta) +\frac{1}{k}\|\theta_k-\theta\|_{\Hx^2}\right)\subset
  \partial\widetilde{J}_N(\theta)+\frac{1}{k}\partial(\|\theta^k-\theta\|_{\Hx^2}),
\end{align*}
which verifies the claim \eqref{eq:substationary}.

In view that $0\in \partial\widehat{J}_{N,k}(\theta^k)$, we have
\begin{align*}
  0\in\eta^k+\mathcal{N}_{\mathbb{U}^{\Px,\Fx}}(\theta^k) +\frac{1}{k}\partial(\|\theta^k-\theta\|_{\Hx^2})|_{\theta=\theta^k},
\end{align*}
where we have used $\eta^k$ instead of $\eta^{\theta^k}$ for short. It follows that there exists $\varphi^k\in\mathcal{N}_{\mathbb{U}^{\Px,\Fx}}(\theta^k)$ and $\|\chi^k\|_{\Hx^2}\leq 1$ that
\begin{equation}\label{eq:pre-projection}
  \eta^k+\varphi^k+\frac{1}{k}\chi^k=0,\qquad \forall~ k\in\N.
\end{equation}

To characterize $\varphi^k$ more precisely, we claim that
\begin{equation}\label{eq:normal-cone-R}
  \mathcal{N}_{\mathbb{U}^{\Px,\Fx}}(\theta^k)=\{y\in\Hx^2:y_t\in N_\Theta(\theta^k_t),\ dt\otimes d\Px\ \textrm{a.e.}\},
\end{equation}
where $N_\Theta(\theta^k_t)$ is the normal cone to $\Theta$ at $\theta^k_t$. It is straightforward to see that the right-hand side of \eqref{eq:normal-cone-R} belongs to $\mathcal{N}_{\mathbb{U}^{\Px,\Fx}}(\theta^k)$ because $y_t\in N_\Theta(\theta^k_t)$ implies that $y_t(z-\theta^k_t)\leq 0$, for $\forall z\in\Theta$ a.e. on $(0,T)\times\Omega$. On the other hand, for any $y\in \mathcal{N}_{\mathbb{U}^{\Px,\Fx}}(\theta^k)$, $z_t=\Pi_\Theta(\theta^k_t+y_t)$ is well-defined by the compactness and convexity of $\Theta$, which yields that there exists $w^z_t\in N_\Theta(z_t)$ such that $z_t+w^z_t=\theta^k_t+y_t$. Note that $z=(z_t)_{t\in[0,T]}$ defined above is an element of ${\mathbb{U}^{\Px,\Fx}}$, and consequently
\begin{align*}
  0\geq&\langle y,z-\theta^k \rangle=\Ex\left[\int_0^T (z_t-\theta^k_t+w^z_t)(z_t-\theta^k_t) dt\right]\\
  =&\|z-\theta^k\|_{\Hx^2}+\Ex\left[\int_0^T w^z_t(z_t-\theta^k_t) dt\right]
  \geq\|z-\theta^k\|_{\Hx^2}.
\end{align*}
It follows that $z=\theta^k$ a.e. on $(0,T)\times\Omega$, which gives that $w^z=y$, i.e. $y_t\in N_\Theta(\theta_t^k),\ dt\times d\Px$ a.e.

Let $\mathbf{X}^{k}$ be the controlled diffusion with $\theta^k$ and $(\mathbf{p}^k,\mathbf{q}^k)$ as the solution to the associated BSDE \eqref{eq:adjoint-SDE}. In view of \eqref{eq:normal-cone-R}, the equation \eqref{eq:pre-projection} can be rewritten by
\begin{align*}
  \theta^k_t+\frac{1}{2\lambda}\tilde{\varphi}^k_t=-\frac{1}{2\lambda}\sum_{i=1}^N u_ip^{k,i}_t-\frac{1}{2\lambda k}\chi^k_t.
\end{align*}
As $\tilde{\varphi}^k_t\in N_\Theta(\theta^k_t)$, the desired result \eqref{eq:theta-k} is verified.
\end{proof}

\end{document}